\documentclass[11pt]{amsart}
\usepackage{tikz}
\usepackage{tikz-cd}
\usepackage{mathrsfs}
\usepackage[mathcal]{euscript}
\usepackage{graphicx}
\usepackage{amsthm,amscd}
\usepackage{calc}
\usepackage{color}

\usepackage{mathrsfs,dsfont}
\usepackage{fancyhdr,amscd}
\usepackage{anysize}
\usepackage{amsmath,amssymb,amsfonts}
\usepackage{epsfig,enumerate}
\usepackage{latexsym}
\usepackage{indentfirst, latexsym}
\usepackage{graphics}
\usepackage[all,poly,knot]{xy}
\usepackage[pagebackref]{hyperref}
\usepackage{lineno}
\usepackage{etoolbox}

\newcommand{\comment}[1]{}

\usepackage{fancyhdr}
\providecommand{\U}[1]{\protect\rule{.1in}{.1in}}

\numberwithin{equation}{section}

\theoremstyle{plain}
\newtheorem{thm}{Theorem}[section]
\newtheorem{lemma}[thm]{Lemma}

\newtheorem{prop}[thm]{Proposition}
\newtheorem{cor}[thm]{Corollary}

\theoremstyle{definition}
\newtheorem{defn}[thm]{Definition}
\newtheorem{ex}[thm]{Example}
\newtheorem{rem}[thm]{Remark}

\newtheorem{claim}{Claim}

\theoremstyle{definition}
\newtheorem{step}{Step}
\AtBeginEnvironment{proof}{\setcounter{step}{0}}

\begin{document}

	\title[Criteria for a fiberwise Fujiki/K\"ahler family  to be locally Moishezon/ projective]{Criteria for a fiberwise Fujiki/K\"ahler family  to be locally Moishezon/projective}

	\author[Jian Chen]{Jian Chen}

	\address{Jian Chen, School of Mathematics and Statistics, Central China Normal University,
		Wuhan 430079, People's Republic of China}
	\email{jian-chen@ccnu.edu.cn}



	\begin{abstract}		
In this paper, 
we utilize the theory of non-K\"ahler loci by S. Boucksom to construct an integral $2$-cohomology class whose restriction to a general fiber is big, and then construct a relatively big line bundle via the exponential sequence. This leads to a local Moishezonness criterion for fibrations whose total spaces are compact complex manifolds in the Fujiki class $\mathscr C$,  generalizing the bimeromorphic version of F. Campana’s local projectivity theorem.  We further combine a similar idea  with the singular Demailly--P\u{a}un theorem by T. Collins--V. Tosatti to obtain a local projectivity criterion for fibrations from compact K\"ahler manifolds,  yielding a new proof and a generalization of F. Campana's local projectivity theorem.       
	\end{abstract}
	\maketitle
    \begingroup
\renewcommand{\thefootnote}{}
\footnotetext{%
© 2026. This manuscript version is made available
under the CC BY-NC-ND 4.0 license:
\url{https://creativecommons.org/licenses/by-nc-nd/4.0/}.}
\endgroup

	\setcounter{tocdepth}{2}
	\tableofcontents

	\section{Introduction}\label{s introduction}

In complex differential geometry and complex algebraic geometry, global positivity plays a crucial role, particularly in the study of deformation families. 	
	Projective morphisms and their birational version, Moishezon morphisms, are fundamental objects of study because they inherently carry such  positivity properties. These morphisms exhibit many interesting and important features (e.g.,  \cite{Ny04}  \cite{Lz04}).  
A concrete motivation for the present paper comes from plurigenera: their deformation invariance is known for smooth locally projective families \cite{Si98}, \cite[Corollary 0.3]{Si02},  \cite{Pa07}, and more generally for locally Moishezon families with canonical singularities \cite[Theorem 1.1]{Tk07}, \cite[Main theorem 1.2]{RT22}. In the K\"ahler setting, Y.-T. Siu conjectured \cite[Conjecture 2.1]{Si02a} that, for a smooth family $\pi:X\to \Delta$ with K\"ahler fibers or with K\"ahler total space, $h^0(X_t,mK_{X_t})$ is independent of $t$ for each positive integer $m$; earlier, M. Levine obtained invariance results for families whose fibers are compact complex manifolds in the Fujiki class $\mathscr C$ under additional assumptions \cite{Lv83, Lv85}, and more recently J. Cao and M. P\u{a}un obtained related results for K\"ahler families \cite{CP23}. These results suggest that plurigenera should remain deformation invariant for families whose fibers are   in the Fujiki class $\mathscr C$ and have canonical singularities, thereby motivating the study on the  local Moishezonness and local projectivity criteria  in the present paper.

We focus in Section \ref{sec-generic-relative-big-line-bundles} on investigating how far a fiberwise Fujiki family  is from being locally Moishezon. 	  We utilize the theory of non-K\"ahler loci by S. Boucksom \cite{B04} to construct an integral $2$-cohomology class whose restriction to a general fiber is big, and then construct a relatively big line bundle via the exponential sequence. This leads to a local Moishezonness criterion for fibrations whose total spaces are compact complex manifolds in the Fujiki class $\mathscr C$, establishing and  generalizing the bimeromorphic version of Campana’s local projectivity theorem in \cite{Ca20} (see Corollary \ref{cor-lagrangi-moishe} and Example  \ref{prop-kummer-flop}).

\begin{thm}[{= Theorem  \ref{thm-local-moishezon-criterion}}]\label{thm-1.1}
Let $f:Z\to S$ be a fibration (see paragraph 3 of \S \ref{section preli}) from a compact complex manifold $Z$ in the Fujiki class $\mathscr C$, with the fiber over $t\in S$ denoted by $Z_t:=f^{-1}(t)$.
Assume that the following conditions hold:
\begin{enumerate}[\rm{(}1\rm{)}]
\item\label{intro-cond-zero-one-fiber}
there exists a point $s\in S_0$ such that the restriction map
\[
        H^0(Z,\Omega_Z^2)\to H^0(Z_s,\Omega_{Z_s}^2)
\]
is zero, where $S_0:=D\cap S_{\rm{reg}}$
(as in the proof of Lemma \ref{lem-zero-one-fiber-zero-all-fibers})  with  $S_{\rm{reg}}$  the smooth locus of $S$ and $D$  the maximal  analytic Zariski open subset of  $S$  such that $f$ is flat and submersive over $D$;
\item\label{intro-cond-torsion-free-r2}
$R^2f_*\mathcal O_Z$ is torsion-free.
\end{enumerate}
Then  $f$ is locally Moishezon (Definition \ref{def-loc-moish}).
\end{thm}

Campana \cite{Ca20} recently established a local projectivity criterion for Lagrangian fibrations from hyperk\"ahler manifolds. Moreover, B. Claudon and A. Höring \cite{CH24} recently established criteria for the global projectivity of morphisms between certain compact K\"ahler spaces. Motivated by these results, we focus in Section \ref{s projective-prime} on investigating how far a fiberwise K\"ahler family is from being locally projective.

We combine an idea similar to the one used above for the  local Moishezonness with the singular Demailly--P\u{a}un  theorem of T. Collins and V. Tosatti to obtain a local projectivity criterion for fibrations from compact K\"ahler manifolds, yielding a new proof and a generalization of Campana's local projectivity theorem.

\begin{thm}[{= Theorem \ref{thm-local-relative-ample-line-bundle}}]\label{thm-intro-relative-ample}
Let $f:Z\to S$ be a fibration from  a compact K\"ahler manifold $Z$, with the fiber over $t\in S$ denoted by $Z_t:=f^{-1}(t)$.
Assume that the following conditions  
hold:
\begin{enumerate}[\rm{(}1\rm{)}]
\item\label{intro-condition-zero-one-fiber-ample} there exists a point $s\in S_0$ such
that the restriction map
\[
        H^0(Z,\Omega_Z^2)\to H^0(Z_s,\Omega_{Z_s}^2)
\]
is zero, where $S_0:=D\cap S_{\rm{reg}}$  with  $S_{\rm{reg}}$  the smooth locus of $S$ and $D$  the maximal  analytic Zariski open subset of  $S$  such that $f$ is flat and submersive over $D$ (note that $S_0$ is analytically Zariski open in $S$);
\item\label{intro-condition-torsion-free-ample} $R^2f_*\mathcal O_Z$ is torsion-free.
\end{enumerate}
Then $f$ is locally projective. 
\end{thm}

Note that the local Moishezonness and local projectivity criteria in this paper are obtained by a unified strategy: we construct an integral $2$-cohomology class and then verify the relevant positivity of its restriction to the fibers. In the Moishezon case, this positivity is bigness and is checked by using the theory of non-K\"ahler loci; in the projective case, it is ampleness and is checked by using the singular Demailly--P\u{a}un theorem of Collins--Tosatti. In particular, in the proof of the local projectivity criterion, the singular Demailly--P\u{a}un theorem is used to construct directly, near each fiber, a K\"ahler form representing the Chern class of the constructed line bundle. This gives a method different from applying the Nakai--Moishezon criterion to check the ampleness on a fiber, whose projectivity is not known, a priori.

	\section{Preliminaries}\label{section preli}
Unless otherwise stated, throughout this paper: all complex analytic
	spaces (equipped with  the complex  topology) are assumed to be Hausdorff and have a countable topology (and are therefore metrizable, paracompact, and countable at infinity, as shown in \cite[51 A. 2 Proposition]{KK83});  all topology notions (e.g., open, closed, dense) are assumed to be 
	w.r.t. the complex  topology; the term   general point/fiber     refers to points/fibers in/over a nonempty Zariski open subset; We often use the generic smoothness theorem
(e.g., \cite[Theorems 1.21, 1.22 and 2.8]{PR94}).
More precisely, if 
$f:X\to S$ is a proper surjective holomorphic map from a complex
manifold onto a reduced irreducible complex analytic space, then
there exists a  nonempty analytic Zariski open subset
$D\subset S$ over which $f$ is flat and submersive; a complex space $X$ is said to be {locally irreducible} if   the local ring $\mathcal{O}_{X,x}$ is an integral domain for any $x\in X$;
  For a cohomology class or a cohomology  group $\bullet$ on $X$, we  denote by $\bullet_{\mathbb R}$ its image in
$H^*(X,\mathbb R)$; We always 
identify $H^2(X,\mathbb R)\cong H^2_{\rm dR}(X,\mathbb R)$ with its image in $H^2(X,\mathbb C)\cong H^2_{\rm dR}(X,\mathbb C)$ when $X$ is a compact complex manifold.     
	
Furthermore, we will repeatedly use without further mention the standard fact
that a connected locally irreducible complex analytic space is
irreducible. In particular, a connected normal complex analytic space
is reduced and irreducible. We also use the fact that complex analytic
spaces are locally connected (e.g., \cite[p. 178]{GR84}); thus, when working locally, an open
neighbourhood may be replaced by its connected component containing
the given point. If the original neighbourhood is Stein, so is this
connected component.

Throughout this paper, by a \emph{fibration} we mean a proper surjective holomorphic map with connected fibers from a normal complex analytic space onto a connected locally irreducible complex analytic space.  In particular, the source space of a fibration is connected (e.g., \cite[Part 1, Lemma 5.7.5]{Stacks}).

    We now  recall some standard notions and results used in the proofs of the present paper.
Let  $X$ be  a compact complex manifold.  Set  $H_{\mathrm{BC}}^{1,1}(X, \mathbb{R})$ to be the real Bott--Chern $(1,1)$-cohomology group:
$$H_{\mathrm{BC}}^{1,1}(X, \mathbb{R}):=\frac{\left\{\theta \in A^{1,1}(X)_{\mathbb{R}} \mid d \theta=0\right\}}{\left\{\sqrt{-1} \partial \bar{\partial} \varphi \mid \varphi \in C^{\infty}(X, \mathbb{R})\right\}}.$$
Let 
$H_{\rm dR}^{p,q}(X)\subset H^{p+q}(X,\mathbb C)$ denote the subspace consisting of
deRham classes represented by $d$-closed $(p,q)$-forms. Set
\[
        H_{\rm dR}^{1,1}(X,\mathbb R)
        :=
        H_{\rm dR}^{1,1}(X)\cap H^2(X,\mathbb R),
\]
where the intersection is taken inside $H^2(X,\mathbb C)$. Note that when
 $X$ satisfies the $\partial\overline{\partial}$-lemma (sometimes referred to as a $\partial\overline{\partial}$-manifold), then   $H_{\mathrm{BC}}^{1,1}(X, \mathbb{R})$   is  canonically isomorphic to  $H_{\mathrm{dR}}^{1,1}(X, \mathbb{R})$.

	\begin{defn}[{e.g., \cite[Definition 1.6]{DP04}}]
     Let  $X$ be  a compact complex manifold.
A \emph{K\"ahler current} on  $X$ is a closed positive current $T$ of bidegree $(1,1)$ which satisfies $T \geq \varepsilon \omega$ for some $\varepsilon>0$ and some smooth  Hermitian form $\omega$ on $X$.
    We refer to a class in $H_{\mathrm{BC}}^{1,1}(X, \mathbb{R})$ that is represented by a K\"ahler current  as a \emph{big class}, and the set (which is clearly a cone) of all big  classes as the \emph{big cone}.   
	\end{defn}
	
	\begin{defn}[{e.g., \cite[Chapter VI,  (12.10) Definition]{Dm12}}]
		Let  $X$ be  a compact connected complex manifold. $X$ is called   \emph{ in the Fujiki class $\mathscr{C}$}  if there exists a   proper modification  $\tilde{X}\to X$ from a 
		compact K\"ahler manifold $\tilde{X}$.  \footnote{By eliminating the indeterminacy, one can   see that a compact complex manifold $X$ is in the Fujiki class $\mathscr{C}$ if and only if  $X$ is bimeromorphic  to a compact K\"ahler manifold. }  In particular, a compact complex manifold in the Fujiki class
$\mathscr C$ satisfies the $\partial\bar\partial$-lemma
(\cite[\S5, (5.22) Theorem,  p. 269]{DGMS75}) and admits the Hodge
decomposition  (e.g., \cite[\S 1(a), p. 194]{Fu83}).
\end{defn}

	By pushing forward a K\"ahler form from a K\"ahler modification of a compact complex manifold in the Fujiki class $\mathscr C$, one obtains a K\"ahler current on the original manifold. We will use this result repeatedly without further citation.
	
	\begin{thm}[{e.g., \cite[Theorem 3.4]{DP04}}]\label{lemma-fujiki}
		A compact connected complex manifold $X$ is in the Fujiki class $\mathscr{C}$ if and only if it admits a big class (or  a K\"ahler current).
	\end{thm}

	\begin{lemma}[{e.g., \cite[Lemma 2.1]{C22}}]\label{big cone open}
Let $X$ be a compact complex manifold in the Fujiki class $\mathscr C$. Then the big cone of $X$ is a
nonempty open subset of $H_{\mathrm{BC}}^{1,1}(X, \mathbb{R})$; via the canonical isomorphism given by the $\partial\bar\partial$-lemma, it may also be regarded as an open cone in $H^{1,1}_{\rm dR}(X,\mathbb R)$.   
	\end{lemma}

\begin{defn}[{(\cite[Definition 3.16]{B04})}]\label{def-nonkah}
		Let $\alpha$  be a big class of a compact complex manifold $X$ in the Fujiki class $\mathscr C$.  The \emph{non-K\"ahler locus} $E_{n K}(\alpha)$ of $\alpha$  is defined to be 
		$$
		E_{n K}(\alpha):=\bigcap_{T \in \alpha} E_{+}(T),
		$$
		where $E_{+}(T)$ denotes the set of points of $X$ such that the K\"ahler current $T$ has positive Lelong numbers, and $T$ ranges over all  K\"ahler current representatives of the class $\alpha$.  Clearly, the non-K\"ahler locus of a K\"ahler class is empty. Moreover,  the non-K\"ahler  locus is an analytic subset of $X$ (\cite[Theorem 3.17-(ii)]{B04}).
	\end{defn}

\begin{thm}[{\cite[Theorem 4.6]{JS93}}]\label{ji-shiff-big}
		Let $L$ be a holomorphic line bundle on a compact complex manifold $M$. Then $L$ is big if and only if $L$ has a singular Hermitian metric (in the sense of Demailly) $h$ such that the curvature current $c_1(L, h)$ is  a K\"ahler current.
	\end{thm}

	\begin{defn}\label{def-proj morp}
		Let $f:X\to Y$ be a proper surjective holomorphic map of complex analytic spaces. We say that $f$ is \emph{projective} if there exists an $f$-ample line bundle on $X$, that is, a line bundle whose restriction to each fiber of $f$ is ample (for many equivalent characterizations, we refer the reader to \cite[p. 25]{Ny04}); we say that $f$ is \emph{locally projective} if for each point $y\in Y$, there exists an open neighborhood $U_y$ of $y$ such that the restricted map $f^{-1}(U_y)\to U_y$ is projective. 
	\end{defn}
	
	\begin{defn}\label{def-loc-moish}
		Let $p:X\to S$ be a proper surjective holomorphic map of complex analytic spaces.  $p$ is said to be \emph{Moishezon} if it is bimeromorphically   equivalent over $S$ to a  projective morphism
		$q:Y\to S$, i.e., there exists a bimeromorphic   map $g:X\dashrightarrow Y$ such that $p=q\circ g$;     $p$ is called \emph{locally Moishezon},  if every point $s\in S$ has an open   neighborhood $W_s$ 
		such that    $p$ is bimeromorphic to a projective morphism over $W_s$.
       \end{defn}

In this paper, we define the smooth family (morphism) as follows.

	\begin{defn}\label{defn-smoothfamily}
		A \emph{smooth family (morphism)} is defined as      
        a submersive fibration  between complex manifolds (not necessarily compact).
        In particular, Ehresmann's theorem applies to a smooth morphism locally.
	\end{defn}

    \begin{rem}
 A smooth family (morphism) in the sense of Definition \ref{defn-smoothfamily}  is automatically flat 
   by the analytic version of the Sard theorem (e.g., \cite[Theorem 1.14 or Proposition 2.15]{PR94}), and is thus a
\emph{smooth (i.e., flat and submersive) morphism} in the standard sense
(e.g., \cite[p. 114]{PR94}).
Note that, for a smooth morphism in the standard sense, the source and the target need not be smooth.
Thus, in the present paper, when we want to obtain a smooth family (morphism)
in the sense of Definition \ref{defn-smoothfamily} from the generic
smoothness theorem, we always intersect the resulting open subset
with the smooth locus of the base.
  \end{rem}
  
We now give the following criteria for torsion-freeness.

\begin{thm}[{\cite[Theorem 24]{Kol22}}]\label{kollarfree-preli}
		Let $g: X \to S$ be a smooth family. Assume that $H^i\left(X_s, \mathbb{C}\right) \rightarrow H^i\left(X_s, \mathcal{O}_{X_s}\right)$ is surjective for every $i\in \mathbb{N}$ for some $s \in S$. Then $R^i g_* \mathcal{O}_X$ is locally free in a neighborhood of $s$ for every $i\in \mathbb{N}$.
	\end{thm}

\begin{lemma}\label{lem-torsion-free-hyperfujiki}  
Let $f:Z\to S$   be a fibration from a compact complex manifold $Z$ in the Fujiki class $\mathscr C$.
Assume that, for each $s\in S$, there exists an open neighbourhood $U$
of $s$ such that $K_Z^{-1}|_{f^{-1}(U)}$ admits a smooth Hermitian metric
with semipositive curvature on $f^{-1}(U)$.
Then $R^qf_*\mathcal O_Z$ is
torsion-free for each $q\geq 0$.
\end{lemma}

\begin{proof}

Fix $s\in S$, and choose an open neighbourhood $U$ of $s$ as in the
assumption.  
After replacing $U$ by its connected component containing $s$, we
may assume that $U$ is 
reduced and irreducible.  
Set $Z_U:=f^{-1}(U)$ and
$f_U:=f|_{Z_U}:Z_U\to U$. Then it suffices to prove that $R^q(f_U)_*\mathcal O_{Z_U}$ is torsion-free for each $q\geq 0$.

Since $f_U$ is proper with connected fibers and $U$ is connected,
$Z_U$ is connected
(e.g., \cite[Part 1, Lemma 5.7.5]{Stacks}).
Since $Z$ is a compact complex manifold in the Fujiki class $\mathscr C$, there exists a proper
modification $\mu:Y\to Z$ from a compact K\"ahler manifold $Y$.

Set $Y_U:=\mu^{-1}(Z_U)$ and
\[
h_U:=f_U\circ\mu|_{Y_U}:Y_U\to U.
\]
Then the K\"ahlerness of $Y_U$ and the properness of $f_U$ and $\mu$  give that $h_U$ is a proper K\"ahler morphism
in the sense of \cite[Definition 6.1]{Tk95}.
Clearly, $f_U$ is bimeromorphic over $U$ to $h_U$.
By assumption, $K_{Z_U}^{-1}=K_Z^{-1}|_{Z_U}$ admits a smooth Hermitian
metric with semipositive curvature. Thus, applying
\cite[Theorem 6.5]{Tk95} to $f_U$ and $E=K_{Z_U}^{-1}$, we obtain that $R^q(f_U)_*\mathcal O_{Z_U}$
is torsion-free for each $q\geq 0$.
\end{proof}

\begin{defn}
    A compact complex manifold $X$ is called a \emph{hyperk\"ahler manifold (or irreducible holomorphic symplectic manifold)} if $X$ is simply connected, K\"ahler, and
$H^0(X,\Omega_X^2)=\mathbb C\sigma$
for some everywhere non-degenerate holomorphic two-form $\sigma$ (in particular, it is even dimensional). 
\end{defn}

We now introduce hyperfujiki
manifolds, which may be viewed as natural bimeromorphic analogues of
hyperk\"ahler manifolds. A standard method to produce non-K\"ahler hyperfujiki manifolds is to utilize a Mukai
flop of a hyperk\"ahler manifold,  (e.g., Example \ref{prop-kummer-flop}, \cite[\S 4.4, Proposition 4.18]{Y01}, \cite[Examples 21.7 and 21.9]{GHJ02}, \cite[Example 5.1]{ACRT18}). 

\begin{defn}\label{def-hyperfujiki}
A \emph{hyperfujiki manifold} is a simply connected compact complex
manifold $(X,\sigma)$ in the Fujiki class $\mathscr C$, where $\sigma$ is an
everywhere nondegenerate holomorphic two-form and
$H^0(X,\Omega_X^2)=\mathbb C\sigma$.  
\end{defn}

Clearly, for a hyperk\"ahler or 
hyperfujiki manifold 
 $(X,\sigma)$ of dimension $2n$, 
$\sigma^{\wedge n}$ trivializes $K_X$. Thus
$K_X\cong \mathcal O_X$. 
Note that any holomorphic $p$-form on a compact complex manifold 
satisfying the $\partial\bar\partial$-lemma is automatically $d$-closed. 
In particular, the holomorphic symplectic form of a hyperk\"ahler or 
hyperfujiki manifold is $d$-closed.

\begin{lemma}\label{hyperfujiki d-closed}
  Let $X$ be a compact complex manifold satisfying the
$\partial\bar\partial$-lemma. Then any holomorphic $p$-form
on $X$ is  $d$-closed.
\end{lemma}

\begin{proof}
For  any $\alpha\in H^0(X,\Omega_X^p)$, we obtain
\[
d(\partial\alpha)
=\bar\partial\partial\alpha
=-\partial\bar\partial\alpha
=0.
\]
That is to say, $\partial\alpha$ is $d$-closed and $\partial$-exact. By the
$\partial\bar\partial$-lemma, it is $\partial\bar\partial$-exact.
However, $\partial\alpha$ has type $(p+1,0)$, so for degree reasons,
a $\partial\bar\partial$-exact form of this type must vanish. Therefore
$\partial\alpha=0$. Together with $\bar\partial\alpha=0$, this gives
$d\alpha=0$.
\end{proof}

\begin{defn}\label{def-lagran}
An $n$-dimensional reduced and irreducible  complex analytic  subspace $Z$ of a $2n$-dimensional hyperk\"ahler/hyperfujiki manifold $(X, \sigma)$ is called \emph{Lagrangian}  if  $\sigma$  restricts to a zero $(2,0)$-form on the smooth part of $Z$.
\end{defn}

\begin{defn}\label{def-Lagran-fibra}
    Let $X$ be a hyperk\"ahler/hyperfujiki 
 manifold. A \emph{Lagrangian fibration} on $X$ is a fibration $f:X\to S$ with $S$ normal such that, for each $s\in S$, each irreducible component of $(f^{-1}(s))_{\rm{red}}$ is  Lagrangian.
\end{defn}

\section{Construction of a holomorphic line bundle}\label{sect 3}

In this section, motivated by \cite[Lemma 1.2]{Ca20}, \cite[Proposition 4.13]{RT21} and 
\cite[Theorem 21]{Kol22}, we use the exponential
 sequence to construct a holomorphic line bundle in a certain
setting.  This construction may be viewed as a 
	relative version of the theorem   on $(1,1)$-classes (e.g.,  the famous Lefschetz theorem for compact K\"ahler manifolds; see \cite[p. 135, Corollary]{MK71} and  \cite[Chapter V, (13.9) Theorem]{Dm12} for more general manifolds).

 We first  give the following observation.

	\begin{lemma}\label{germ torsion free}
		Let $X$ be a connected and locally irreducible   complex analytic space  and $F$ a coherent sheaf on $X$.  Let $U$ be a dense (in  the usual complex  topology)   subset of $X$. Assume that $s\in \Gamma(X, F)$ satisfies that $s_x\in \rm{Tor}(F)_x$   for any   $x\in U$.
		Then $s_x\in \rm{Tor}(F)_x$ for any $x\in X$, where $ \rm{Tor}(\bullet)$ is the torsion sheaf of $\bullet$.  
	\end{lemma}

    \begin{proof}
Consider the natural composite morphism
\[
    \varphi:\mathcal{O}_X
    \xrightarrow{\,1\mapsto s\,} F
    \longrightarrow F/\operatorname{Tor}(F),
\]
and apply the identity theorem (e.g., \cite[Lemma 9.2]{RS17})
to the section $\varphi(1)$ of the torsion-free coherent sheaf
$F/\operatorname{Tor}(F)$.
\end{proof}

We now prove the existence of a holomorphic line bundle with prescribed first
Chern class.
    
	\begin{thm}\label{rela 1-1 class}
		Let $f: X \to  S$ be a proper surjective holomorphic map 
		with connected fibers from a reduced  complex analytic space  $X$ to a
		locally irreducible and connected   Stein space $S$,  with $X_t$ denoting the fiber of $f$ over $t\in S$.  Assume that there exists an  open    dense subset $D\subseteq S$  such that $X_t$ is smooth for any $t\in D$ and that $R^2 f_*\mathcal{O}_X$ is locally  free on $D$.  Assume moreover that the following conditions hold:
  		\begin{enumerate}
			\item\label{condition restric-new}
			There exists a class $u\in H^2(X, \mathbb{Z})$  such that  for any $t\in D$, $u|_{X_t}$ is of type $(1,1)$, i.e., the image of 
			$u|_{X_t}$ in  $H^2(X_t, \mathbb{C})$ can be represented by a $d$-closed $(1,1)$-form (e.g., \cite[p. 135, Corollary]{MK71}). 
			\item
			$R^2 f_*\mathcal{O}_X$ is torsion free. \footnote{Note that if the base space $S$ degenerates to a simple point,  the   torsion-freeness condition holds automatically. So  Theorem \ref{rela 1-1 class} can be regarded as a 
				relative version of the theorem   on $(1,1)$-classes.} 
		\end{enumerate}
		Then there exists a holomorphic line bundle $L$ on $X$ such that $c_1(L)=u$.
		
	\end{thm}

	\begin{proof}
First note that 
		we have a natural morphism $
		H^2\left(X, \mathbb{Z}\right) \rightarrow H^2\left(X, \mathcal{O}_{X}\right)$, derived from the natural morphism  $\mathbb{Z}\to \mathcal{O}_{X}$ of sheaves.  
		Since $f$ is proper, $R^2 f_*\mathcal{O}_{X}$ is coherent by the Grauert's direct image theorem. By the theorem of Cartan B, one can use the Leray spectral sequence to get that $ H^2\left(X, \mathcal{O}_{X}\right)\cong H^0\left(S, R^2 f_*\mathcal{O}_{X}\right)$, and thus we now have the morphism
		\begin{equation}\label{isomorphism by Leray-new}
			H^2\left(X, \mathbb{Z}\right) \rightarrow H^2\left(X, \mathcal{O}_{X}\right)\cong H^0\left(S, R^2 f_*\mathcal{O}_{X}\right).
		\end{equation}
		Let now  $u^{\prime} \in H^0\left(S, R^2 f_*\mathcal{O}_{X}\right)$ be the image of $u$ under the map \eqref{isomorphism by Leray-new}.    
		
		\begin{claim}\label{claim-u'-new}
			$u^{\prime}|_D=0$ as an element in $\Gamma(D, R^2 f_*\mathcal{O}_{X})$.  
		\end{claim}
		
		\begin{proof}
			Since $f$ is proper,   $S$ is reduced and $X$  has a countable topology,  $f$ is generally flat by the analytic version of the generic flatness theorem (e.g., \cite[Theorem 2.8]{PR94}), i.e., $f$ is flat over a nonempty  analytic Zariski open subset of $S$.  Furthermore, the coherent sheaves $R^p f_* \mathcal{O}_{X}$ are locally free generally for any $p\geq 2$  (e.g., \cite[Proposition 7.17]{Re94}) by the reducedness of $S$.      Then $h^p(X_t, \mathcal{O}_{X_t})$ is locally constant for any $p\geq 2$ on certain Zariski open subset $\tilde{S}$ of $S$, based on \cite[Corollary 3.10+ Theorem 4.12]{BS76}.

For any $t\in D$, we denote by $u^{\prime}(t)$  the value of $u^{\prime}$ at $t$, i.e., the image of the germ  $(u^{\prime})_t$ under the map $$(R^2 f_* \mathcal{O}_{X})_t\to (R^2 f_* \mathcal{O}_{X})_t / \mathfrak{m}_t (R^2 f_* \mathcal{O}_{X})_t,$$
where $\mathfrak{m}_t$ is the maximal ideal of $\mathcal{O}_{S,t}$. Note that  the base change map
\begin{equation}\label{iso-base-change}
    (R^2 f_* \mathcal{O}_{X})_t / \mathfrak{m}_t (R^2 f_* \mathcal{O}_{X})_t \to H^2(X_t, \mathcal{O}_{X_t})
\end{equation}
			is isomorphic for any  $t\in \tilde{S}$,  by  Grauert's base change theorem (e.g.,  \cite[p. 33-(8.5) Theorem]{BHPV04}).

Since $D$ is reduced, $R^2f_*\mathcal{O}_X$ is locally free on $D$,
and $D\cap\tilde{S}$ is dense in $D$, it suffices to prove that 
$u'(t)=0$ for each $t\in D\cap\tilde{S}$.
Let $t\in D\cap\tilde{S}$ be any fixed point.		Consider the long exact sequence  associated to the following diagram (based on cohomological properties on the sheaf of extension by zero, e.g., \cite[Chapter III, Lemma 2.10]{Ha77}) which is induced by the  exponential  sequences (e.g., \cite[p. 246, 54.3 Lemma]{KK83} for a complex analytic space) on  $X$ and  $X_t$,
			\begin{center}
				\begin{tikzpicture}
					\node (A) at (0,1) {$0$};
					\node (B) at (1.5,1) {$\mathbb{Z}$};
					\node (C) at (3,1) {$\mathcal{O}_{X}$};
					\node (D) at (4.5,1) {$\mathcal{O}^*_{X}$};
					\node (E) at (6,1) {$1$};
					\node (a) at (0,0) {$0$};
					\node (b) at (1.5,0) {$\iota_*\mathbb{Z}$};
					\node (c) at (3,0) {$\iota_*\mathcal{O}_{X_t}$};
					\node (d) at (4.5,0) {$\iota_*\mathcal{O}^*_{X_t}$};
					\node (e) at  (6,0) {$1$};
					
					\draw[->] (A) -- (B)  node[midway, above]{};
					\draw[->] (B) -- (C)  node[midway, above]{};
					\draw[->] (C) -- (D)  node[midway, above]{};
					\draw[->] (D) -- (E)  node[midway, above]{};
					\draw[->] (a) -- (b)  node[midway, below]{};
					\draw[->] (b) -- (c)  node[midway, below]{};
					\draw[->] (c) -- (d)  node[midway, below]{};
					\draw[->] (d) -- (e)  node[midway, below]{};
					\draw[->] (B) -- (b)  node[midway, left]{};	
					\draw[->] (C) -- (c)  node[midway, left]{};	
					\draw[->] (D) -- (d)  node[midway, left]{};		
				\end{tikzpicture}
			\end{center}
			where $\iota: X_t \hookrightarrow X$ is the natural inclusion.  
Since $X_t$ is  compact and the integral class
$u|_{X_t}\in H^2(X_t,\mathbb{Z})$ is of type $(1,1)$, there exists
a holomorphic line bundle $L_t$ on $X_t$ such that
$c_1(L_t)=u|_{X_t}$ (e.g., \cite[p. 135, Corollary]{MK71}).  By exactness of the long cohomology sequence associated with
the exponential sequence on $X_t$, the class $u|_{X_t}$ maps to zero
in $H^2(X_t,\mathcal{O}_{X_t})$.
It then follows that the image of $u'(t)$ under the base change
isomorphism \eqref{iso-base-change} is zero.
Thus $u'(t)=0$, as required.  This completes the proof of Claim \ref{claim-u'-new}.
	\end{proof}
		
		By Claim \ref{claim-u'-new} and the  torsion-freeness of $R^2 f_* \mathcal{O}_{X}$, $u^{\prime}$ vanishes everywhere on $S$, based on Lemma \ref{germ torsion free}.  Consequently, by again the long exact sequence associated to the  exponential exact sequence over $X$, we obtain that $u$ is the first Chern class of a holomorphic line bundle ${L}$ on $X$.  This completes the proof of Theorem \ref{rela 1-1 class}.
	\end{proof}

\section{Relatively big line bundles and local Moishezonness}\label{sec-generic-relative-big-line-bundles}

In this section,  we utilize the theory of non-K\"ahler loci by S. Boucksom to construct an integral $2$-cohomology class whose restriction to a general fiber is big, and then construct a relatively big line bundle (i.e., whose restriction to a general fiber is big) utilizing the construction of holomorphic line bundles in Section \ref{sect 3}. This leads   to a local Moishezonness criterion for fibrations whose total spaces are compact complex manifolds in the Fujiki class $\mathscr C$,  generalizing the bimeromorphic version of Campana’s local projectivity theorem.  We then adapt Matsushita’s argument to show that fibrations  from hyperfujiki manifolds onto lower-dimensional normal compact K\"ahler spaces    are Lagrangian. In particular,  such fibrations are locally Moishezon.

\subsection{A  local Moishezonness criterion}

We first show that for a proper morphism from a compact complex manifold in the Fujiki class $\mathscr C$, any big class behaves well on general smooth fibers.

\begin{lemma}\label{lem-big-class-restricts-to-general-fiber-zariski}
Let $f:Z\to S$ be a fibration from a compact complex manifold $Z$ in the Fujiki class $\mathscr C$, with the fiber over $t\in S$ denoted by $Z_t:=f^{-1}(t)$, and let $\beta\in H_{\rm{BC}}^{1,1}(Z,\mathbb R)$ be a big class.
Then there exists an analytic Zariski open subset $D_{\beta}\subseteq S$ such that
$f^{-1}(D_{\beta})\to D_{\beta}$ is smooth (Definition \ref{defn-smoothfamily}) and
$Z_t\not\subseteq E_{\rm{nK}}(\beta)$ for each $t\in D_{\beta}$.
As a result, $\beta|_{Z_t}$ is a big class on $Z_t$ for each $t\in D_{\beta}$.
\end{lemma}

\begin{proof}
By elementary analysis of thin analytic subsets of a complex analytic space,   the following assertion can be readily obtained.

\begin{claim}
   Let $A\subsetneq Z$ be a proper analytic
subset. Then there exists an analytic Zariski open subset $U_A\subseteq S$ such 
that $ Z_t\not\subseteq A$
for each $t\in U_A$. 
\end{claim}

Now take $A$ in the above claim to be the non-K\"ahler locus
$E_{\rm{nK}}(\beta).$
Since $\beta$ is big,  $E_{\rm{nK}}(\beta)$ is a proper analytic
subset of $Z$ by \cite[Theorem 3.17-(ii)]{B04}. Hence the above claim gives an
analytic Zariski open subset $U_{\beta}\subseteq S$ such that
$ Z_t\not\subseteq E_{\rm{nK}}(\beta)$
for each $t\in U_{\beta}$.

On the other hand, after intersecting with the smooth locus of the base, 
there exists an analytic Zariski open subset $D_{\rm{sm}}\subseteq S$ such that
$ f^{-1}(D_{\rm{sm}})\to D_{\rm{sm}}$
is smooth (Definition \ref{defn-smoothfamily}). Define $D_{\beta}:=U_{\beta}\cap D_{\rm{sm}}$. 

It remains to prove that $\beta|_{Z_t}$ is big for each $t\in D_{\beta}$. Fix
$t\in D_{\beta}$. By \cite[Theorem 3.17-(ii)]{B04}, there exists a K\"ahler current
$T\in\beta$ with analytic singularities whose singular locus is precisely
$E_{\rm{nK}}(\beta)$. Choose a Hermitian form $\omega_Z$ on $Z$ and a real number
$\varepsilon>0$ such that
$ T\geq \varepsilon\omega_Z$.
Since $f^{-1}(D_{\beta})\to D_{\beta}$ is smooth, the fiber $Z_t$ is a smooth (and  connected). Because $Z_t\not\subseteq E_{\rm{nK}}(\beta)$, the local plurisubharmonic
potentials of $T$ are not identically $-\infty$ on $Z_t$. Thus the restriction
$T|_{Z_t}$ is well defined. Furthermore,  $T|_{Z_t}$ is a  
K\"ahler current ($T|_{Z_t}\geq \varepsilon \omega_Z|_{Z_t}$) representing
$\beta|_{Z_t}$.
Therefore $\beta|_{Z_t}$ is a big class on $Z_t$.
\end{proof}

The following result should be well known, although we were unable to find a suitable reference. We therefore include a proof in the appendix for the reader’s convenience.

\begin{lemma}[= Lemma \ref{lem-big-chern-big-line-appen}]
\label{lem-big-chern-big-line}
Let $M$ be a compact complex manifold in the Fujiki class $\mathscr C$ and let $L$ be a holomorphic line bundle
on $M$. If $(c_1(L))_{\mathbb R}$  is a big class, then $L$ is a big line bundle. 
\end{lemma}

\begin{lemma}\label{lem-zero-one-fiber-zero-all-fibers}
Let $f:Z\to S$ be a fibration from a compact complex manifold $Z$ in the Fujiki class $\mathscr C$, with the fiber over $t\in S$ denoted by $Z_t:=f^{-1}(t)$.
Then there exists a connected  analytic Zariski open subset $S_0\subseteq S$  such that
\[
        f_0:=f|_{f^{-1}(S_0)}:f^{-1}(S_0)\to S_0
\]
is smooth (Definition \ref{defn-smoothfamily}). For such an $S_0$, the following conditions are equivalent:
\begin{enumerate}[\rm{(}1\rm{)}]
\item\label{one-each-1} for each $t\in S_0$, the restriction map
\[
        H^0(Z,\Omega_Z^2)\longrightarrow H^0(Z_t,\Omega_{Z_t}^2)
\]
is zero;
\item\label{one-each-2} there exists some $t_0\in S_0$ such that the restriction map
\[
        H^0(Z,\Omega_Z^2)\longrightarrow H^0(Z_{t_0},\Omega_{Z_{t_0}}^2)
\]
is zero.
\end{enumerate}
\end{lemma}

\begin{proof}
Note that there exists  a maximal analytic Zariski open subset $D\subseteq S$ such that
\[
        f|_{f^{-1}(D)}:f^{-1}(D)\to D
\]
 is flat and submersive (which is called smooth in \cite[p. 114]{PR94}).
Set $S_0:=D\cap S_{\rm{reg}}$, where $S_{\rm{reg}}$ is the smooth locus of $S$.  Then  $f|_{f^{-1}(S_0)}$ is smooth in the sense of Definition \ref{defn-smoothfamily}.  Note that  $S_0$ is a connected analytic Zariski open subset of $S$.  Set
\[
        Z_0:=f^{-1}(S_0)
        \quad\text{and}\quad
        f_0:=f|_{Z_0}:Z_0\to S_0.
\]
By the definition of $S_0$,  $f_0$ is a proper holomorphic submersion between complex manifolds.

Clearly, \eqref{one-each-1} implies \eqref{one-each-2}. Conversely, assume that
\eqref{one-each-2} holds. It suffices to prove that, for any fixed
$\eta\in H^0(Z,\Omega_Z^2)$, one has $\eta|_{Z_t}=0$ for each $t\in S_0$.

Since $Z$ is a compact complex manifold in the Fujiki class $\mathscr C$, it satisfies the $\partial\bar{\partial}$-lemma. By Lemma \ref{hyperfujiki d-closed},  the form $\eta$ is
$d$-closed on $Z$, and hence its restriction (still denoted by $\eta$) to $Z_0$ is also $d$-closed. Thus $\eta$
defines a class
$[\eta]\in H^2(Z_0,\mathbb C).$
For any $t\in S_0$, the class $[\eta]|_{Z_t}\in H^2(Z_t,\mathbb C)$ is the restriction
of the fixed cohomology class $[\eta]$. Since $f_0$ is a proper holomorphic submersion,
it follows from Ehresmann's theorem that $ R^2(f_0)_*\mathbb C$
is a local system on $S_0$. Moreover, the fiberwise restrictions of $[\eta]$ define a
section  $s_{\eta}\in H^0(S_0,R^2(f_0)_*\mathbb C)$ such that $s_{\eta}(t):=\alpha_t((s_{\eta})_t)=[\eta]|_{Z_t}$, 
where $\alpha_t$ is the canonical isomorphism  $(R^2(f_0)_*\mathbb C)_t\cong H^2(Z_t, \mathbb{C})$.
Indeed, locally on $S_0$, Ehresmann's theorem identifies the cohomology groups
$H^2(Z_t,\mathbb C)$ of the fibers, and under this identification the classes
$[\eta]|_{Z_t}$ are induced by the same class on the local total space. Thus
$s_{\eta}$ is locally constant.  

By the given assumption, $\eta|_{Z_{t_0}}=0$.  Thus we obtain
\[
        s_{\eta}(t_0)=[\eta]|_{Z_{t_0}}=0.
\]
Since $S_0$ is connected and $s_{\eta}$ is locally constant, it follows that
$s_{\eta}=0$. Therefore $ [\eta]|_{Z_t}=0$
for each $t\in S_0$.

Since for each $t\in S_0$,  $Z_t$  is also a  compact complex manifold in the Fujiki class $\mathscr C$  (e.g., \cite[Lemma 4.6]{Fu78/79}),   $Z_t$ satisfies the
$\partial\bar{\partial}$-lemma. Consequently,  the natural map
$H^0(Z_t,\Omega_{Z_t}^2)\to H^2(Z_t,\mathbb C)$
is injective. It then follows from $[\eta]|_{Z_t}=0$ that
$ \eta|_{Z_t}=0$
for each $t\in S_0$. This proves \eqref{one-each-1}.
\end{proof}

Before stating the local Moishezonness criterion, we first recall how a relatively big line bundle  (i.e., restriction to general fiber is big)  gives rise to a Moishezon morphism.

\begin{lemma}[{\cite[Definition 18, Lemma 19]{Kol22}}]
\label{lem-kollar-very-big-locus}
Let $g:X\to S$ be a proper morphism of normal irreducible complex analytic
spaces, and let $M$ be a holomorphic line bundle on $X$. Set the very big (\cite[Definition 10-(4)]{Kol22}) locus to be
\[
        \operatorname{VB}_S(M)
        :=
        \{s\in S\mid M|_{X_s}\text{ is very big  on }X_s\}.
\]
Then 
$\operatorname{VB}_S(M)$ is either nowhere dense in the analytic Zariski topology,
or contains a dense open subset of $S$. In the latter case, $g:X\to S$ is a
Moishezon morphism.
\end{lemma}

\begin{prop}
\label{prop-general-fiber-big-line-bundle-moishezon}
Let $g:X\to S$ be a proper surjective morphism of normal irreducible complex analytic
spaces. Let $L$ be a holomorphic line bundle on $X$. Assume that there exists a
nonempty open (in the usual complex topology) subset $D\subset S$  such that
 $X_s$ is normal and  $L|_{X_s}$ is big for each $s\in D$. Then $g:X\to S$ is a Moishezon morphism.
\end{prop}

\begin{proof}
For each positive integer $m$, set
$V_m:=\operatorname{VB}_{S}(L^m)$.
Since $X_s$ is normal and  $L|_{X_s}$ is big for each $s\in D$, for each $s\in D$ there exists a
positive integer $m=m(s)$ such that $L^m|_{X_s}$ is very big (e.g., (\cite[Definition 10-(4)]{Kol22})).   Thus
\[
        D\subset \bigcup_{m\geq 1}V_m.
\]

We claim that there exists some positive integer $m$ such that $V_m$ is not nowhere dense in the analytic Zariski topology. Suppose, to the contrary, that $V_m$ is nowhere
dense in the analytic Zariski topology for each positive integer $m$.  
Note that if a subset $A\subset S$ is nowhere dense in the analytic
Zariski topology, then $A$ is nowhere dense in the usual complex topology. 
Then each $V_m$ is nowhere dense in the complex
topology. 
Note that every locally compact Hausdorff space is a Baire space.  A complex space with
its complex topology is locally compact and Hausdorff. 
 Consequently, $S$ is
a Baire space, and thus every countable union of closed sets with empty interior has empty interior. Consequently,  $\bigcup_{m\geq 1}\overline{V_m}$
 (the closure is w.r.t. the usual complex topology) has empty interior, which contradicts the fact that 
$\bigcup_{m\geq 1}\overline{V_m}$ contains a nonempty open subset $D$.
The claim follows.

Now choose a positive integer $m_0$ such that $V_{m_0}$ is not nowhere dense in the
analytic Zariski topology. Applying Lemma
\ref{lem-kollar-very-big-locus} to the line bundle $L^{m_0}$, we obtain that
$g:X\to S$ is a Moishezon
morphism.
\end{proof}

We now establish the following criterion for a fibration to be locally Moishezon, via the construction of a relatively big line bundle over a neighborhood of a point in the base.  In particular, it also yields a criterion for deciding whether a degeneration of Moishezon varieties remains Moishezon (as shown in Remark \ref{rem-Moishe degene}), a question which seems to be of independent interest.

\begin{thm}
\label{thm-local-moishezon-criterion}
Let $f:Z\to S$ be a fibration from a compact complex manifold $Z$ in the Fujiki class $\mathscr C$, with the fiber over $t\in S$ denoted by $Z_t:=f^{-1}(t)$.
Assume that the following conditions hold:
\begin{enumerate}[\rm{(}1\rm{)}]
\item\label{cond-zero-one-fiber}
there exists a point $s\in S_0$ such that the restriction map
\[
        H^0(Z,\Omega_Z^2)\to H^0(Z_s,\Omega_{Z_s}^2)
\]
is zero, where $S_0:=D\cap S_{\rm{reg}}$
(as in the proof of Lemma \ref{lem-zero-one-fiber-zero-all-fibers})  with  $S_{\rm{reg}}$  the smooth locus of $S$ and $D$  the maximal  analytic Zariski open subset of  $S$  such that $f$ is flat and submersive over $D$;
\item\label{cond-torsion-free-r2}
$R^2f_*\mathcal O_Z$ is torsion-free.
\end{enumerate}
Then  $f$ is locally Moishezon (Definition \ref{def-loc-moish}).  
\end{thm}

\begin{proof}
Note that  $S_0$ is a nonempty analytic
Zariski open subset of $S$.  By condition
\eqref{cond-zero-one-fiber} and Lemma \ref{lem-zero-one-fiber-zero-all-fibers}, the restriction map
\[
        H^0(Z,\Omega_Z^2)\longrightarrow H^0(Z_\tau,\Omega_{Z_\tau}^2)
\]
is zero for each $\tau\in S_0$.  Motivated by an insightful observation of \cite[p. 589, paragraph -1]{Ca20} on hyperk\"ahler manifolds and Lagrangian fibrations, we now construct certain integral $2$-class on $Z$.

\begin{step}\label{step-cons of 2 cla}
Construction of an integral $2$-class  whose restriction to a general fiber is big
\end{step}

Since $Z$ is a compact complex manifold in the Fujiki class $\mathscr C$, it satisfies the $\partial\bar{\partial}$-lemma. Thus we have the
 decomposition  
\[
        H^2(Z,\mathbb C)
        =
        H^{2,0}_{\rm dR}(Z)\oplus H^{1,1}_{\rm dR}(Z)\oplus H^{0,2}_{\rm dR}(Z)
\]
and the induced continuous projection 
\[
        p_{1,1}:H^2(Z,\mathbb R)\to H^{1,1}_{\rm dR}(Z,\mathbb R),
\]
where the notations are as defined in Section \ref{section preli}.

 Since $Z$ is a compact complex manifold in the Fujiki class $\mathscr C$, it admits
a big class. Take a big class $\beta_0\in H^{1,1}_{\rm dR}(Z,\mathbb R)$. Since the big cone is
open, there exists an open neighbourhood $V$ of $\beta_0$ in
$H^{1,1}_{\rm dR}(Z,\mathbb R)$ such that each class in $V$ is big. Then $p_{1,1}^{-1}(V)$ is an
open subset of $H^2(Z,\mathbb R)$. Since $H^2(Z,\mathbb Q)_{\mathbb R}$ is dense in
$H^2(Z,\mathbb R)$, we may choose a  class $a\in H^2(Z,\mathbb Q)$ such that
$p_{1,1}(a_{\mathbb R})\in V$. After multiplying $a$ by a positive integer, we obtain a class
$$u\in H^2(Z,\mathbb Z)$$ such that
$ \beta:=p_{1,1}(u_{\mathbb R})$ is big, because the big cone is a positive cone. Moreover,
\[
        u_{\mathbb R}-\beta
        \in
        \bigl(H^{2,0}_{\rm dR}(Z)\oplus H^{0,2}_{\rm dR}(Z)\bigr)\cap H^2(Z,\mathbb R),
\]
where the intersection is taken inside $H^2(Z,\mathbb C)$.

By Lemma \ref{lem-big-class-restricts-to-general-fiber-zariski}, there exists a dense
analytic Zariski open subset $D_{\beta}'\subseteq S$ such that
$ f^{-1}(D_{\beta}')\to D_{\beta}'$
is smooth and $\beta|_{Z_\tau}$ is a big class for each $\tau\in D_{\beta}'$. Set $ D_{\beta}:=D_{\beta}'\cap S_0.$
 Then $D_{\beta}$ is an analytic
Zariski open subset of $S$ such that 
$f^{-1}(D_{\beta})\to D_{\beta}$
is smooth, and  $\beta|_{Z_\tau}$ is big for each $\tau\in D_{\beta}$. Furthermore, 
$\eta|_{Z_\tau}=0$
for each $\eta\in H^0(Z,\Omega_Z^2)$ and each $\tau\in D_{\beta}$.

Since $Z$ satisfies the $\partial\overline{\partial}$-lemma, it 
 follows from Lemma \ref{hyperfujiki d-closed}  that any holomorphic $p$-form on $Z$ is $d$-closed.  Hence, for any $\eta\in H^0(Z,\Omega_Z^2)$,  we have a well-defined class $[\eta]\in   H^{2,0}_{\rm dR}(Z)$. Furthermore, $H^0(Z,\Omega_Z^2)$ can be viewed as a   subspace of $H^2(Z,\mathbb C)$ via the canonical isomorphism $H^0(Z,\Omega_Z^2)\cong H^{2,0}_{\rm dR}(Z)$ given by the $\partial\overline{\partial}$-lemma.
Moreover, the complex conjugation exchanges $H^{2,0}_{\rm dR}(Z)$ and $H^{0,2}_{\rm dR}(Z)$, and  $H^2(Z,\mathbb R)$ is precisely the conjugation-invariant part of $H^2(Z,\mathbb C)$. Hence, any element $\gamma$ of
$$
\left(H^{2,0}_{\rm dR}(Z) \oplus H^{0,2}_{\rm dR}(Z)\right) \cap H^2(Z, \mathbb{R})
$$
can  split  uniquely as $\gamma=\alpha+\overline{\alpha}$ with $\alpha\in H^{2,0}_{\rm dR}(Z)$.

Let $\eta_1,\ldots,\eta_m$ be a  basis of  $H^0(Z,\Omega_Z^2)$.
Then  the complex deRham cohomology classes $[\eta_1],\ldots,[\eta_m]$ form a basis of $H^{2,0}_{\rm dR}(Z)$.  It then follows that this real vector space
\[
        \bigl(H^{2,0}_{\rm dR}(Z)\oplus H^{0,2}_{\rm dR}(Z)\bigr)\cap H^2(Z,\mathbb R)
\]
is spanned by $[\operatorname{Re}\eta_j]$ and $[\operatorname{Im}\eta_j]$, $1\leq j\leq m$, where 
 the real parts $\operatorname{Re}\eta_j=\frac{\eta_j+\overline{\eta_j}}{2}$ and imaginary parts $\operatorname{Im}\eta_j=\frac{\eta_j-\overline{\eta_j}}{2\sqrt{-1}}$
are clearly real $d$-closed $2$-forms and thus $[\operatorname{Re}\eta_j]$ and
$[\operatorname{Im}\eta_j]$ are well-defined real deRham cohomology classes in
$H^2(Z,\mathbb R)$.

Recall that $\eta|_{Z_\tau}=0$
for each $\eta\in H^0(Z,\Omega_Z^2)$ and each $\tau\in D_{\beta}$.
Then 
$ (u_{\mathbb R}-\beta)|_{Z_\tau}=0$
for each $\tau\in D_{\beta}$. Thus
$u_{\mathbb R}|_{Z_\tau}=\beta|_{Z_\tau}$
for each $\tau\in D_{\beta}$.

\begin{step}\label{step-cons of line}
Construction of a relatively big line bundle
\end{step}
For any fixed $t\in S$, choose a connected   Stein open neighbourhood
$U\subseteq S$ of $t$.
Set
\[
        Z_U:=f^{-1}(U)
        \quad\text{and}\quad
        f_U:=f|_{Z_U}:Z_U\to U .
\]
 By condition
\eqref{cond-torsion-free-r2},  the sheaf
$R^2(f_U)_*\mathcal O_{Z_U}$
is torsion-free. Let $D_{\rm{lf}}\subseteq U$ be its locally free locus.  Set
\[
        D_U:=U\cap D_{\beta}\cap D_{\rm{lf}}.
\]
Then $D_U$ is a nonempty analytic Zariski open subset of $U$. For each $\tau\in D_U$, the
fiber $Z_\tau$ is smooth, and $u_{\mathbb R}|_{Z_\tau}=\beta|_{Z_\tau}$; in particular, $u_{\mathbb R}|_{Z_\tau}$ is of type $(1,1)$.

Now we apply Theorem \ref{rela 1-1 class}  for
$f_U:Z_U\to U$. Then there exists a holomorphic line bundle
$L_U$ on $Z_U$ such that $  c_1(L_U)=u|_{Z_U}.$
Finally, for each $\tau\in D_U$, we have
$c_1(L_U|_{Z_\tau})=u|_{Z_\tau}$
and thus
\[
c_1(L_U|_{Z_\tau})_{\mathbb R}
        =
        u_{\mathbb R}|_{Z_\tau}
        =
        \beta|_{Z_\tau}.
\]
Recalling the construction of  
$D_U$ and $D_{\beta}$, we obtain that the class $\beta|_{Z_\tau}$ is big. Therefore, by Lemma
\ref{lem-big-chern-big-line}, the line bundle $L_U|_{Z_\tau}$ is big.

\begin{step}\label{step-stein-moishezon}
Stein factorization and local Moishezonness
\end{step}
Note that $Z_U$ is smooth and connected (e.g., \cite[Part 1, Lemma 5.7.5]{Stacks}).
Take the Stein factorization 
\[
        Z_U\xrightarrow{\,g\,}S_U\xrightarrow{\,\nu\,}U,
\]
with $f_U=\nu\circ g$,  where $g$ is proper and surjective with connected fibers, $\nu$ is finite, and   $S_U$ is normal and irreducible.  By the uniqueness of the Stein factorization applied over $D_U$, 
$\nu$ is an isomorphism over $D_U$.
 Consequently, for each $s'\in\nu^{-1}(D_U)$, the fiber $g^{-1}(s')$ is smooth and $L_U|_{g^{-1}(s')}$ is big. Proposition \ref{prop-general-fiber-big-line-bundle-moishezon}, applied to $g$ and $L_U$, shows that $g$ is a Moishezon morphism.

By Definition \ref{def-loc-moish}, we may choose a projective morphism $q:Y\to S_U$ bimeromorphically equivalent to $g$ over $S_U$. Since $\nu$ is finite,  $\nu\circ q$ is also projective, and thus   $f_U$ is a Moishezon morphism. Since $t\in S$ was arbitrary, $f$ is locally Moishezon.  This completes the proof of Theorem \ref{thm-local-moishezon-criterion}.
\end{proof}

As a direct consequence, we obtain the following result in a simple case.

\begin{cor}
\label{cor-local-moishezon-criterion-h20-zero}
Let $f:Z\to S$ be a fibration from a compact complex manifold $Z$ in the Fujiki class $\mathscr C$, with the fiber over $t\in S$ denoted by $Z_t:=f^{-1}(t)$.
Assume that, for each $s\in S$, there exists an open neighbourhood $U$
of $s$ such that $-K_Z$ is semi-positive on $f^{-1}(U)$.  
Assume moreover that
there exists a point $s\in S_0$ such that
$h^{2,0}(Z_s)=0,$
where $S_0:=D\cap S_{\rm{reg}}$, with $S_{\rm{reg}}$ the smooth locus of $S$
and $D$ the maximal analytic Zariski open subset of $S$ over which $f$ is flat and submersive.
Then $f$ is locally Moishezon.
\end{cor}

\begin{proof}
Since $h^{2,0}(Z_s)=0$, the restriction map
\[
H^0(Z,\Omega_Z^2)\longrightarrow H^0(Z_s,\Omega_{Z_s}^2)
\]
is zero. By Lemma \ref{lem-torsion-free-hyperfujiki},
$R^2f_*\mathcal O_Z$ is torsion-free.
Therefore Theorem \ref{thm-local-moishezon-criterion} applies and shows that
$f$ is locally Moishezon.
\end{proof}

\begin{rem}\label{rem-Moishe degene}
Note that, in Corollary \ref{cor-local-moishezon-criterion-h20-zero}, 
since  complex submanifolds of a compact complex manifold in the Fujiki class $\mathscr C$ are again compact complex manifolds in the Fujiki class $\mathscr C$
 (\cite[Lemma 4.6]{Fu78/79}), each smooth fiber is a compact complex manifold in the Fujiki class $\mathscr C$. Suppose that $h^{2,0}(Z_s)=0$ for some point $s\in S_0$. Since Hodge numbers are invariant in smooth families with  fibers being compact complex manifolds in the Fujiki class $\mathscr C$, we have $h^{2,0}(Z_t)=0$ for each $t\in S_0$.
It then follows from the definition of a compact complex manifold in the Fujiki class $\mathscr C$, Kodaira's projectivity criterion for compact K\"ahler manifolds, and the fact that the $(0,q)$-Hodge numbers are bimeromorphic invariants, that $Z_t$ is Moishezon for each $t\in S_0$. In other words, in Corollary \ref{cor-local-moishezon-criterion-h20-zero}, the morphism $f$ has Moishezon general fibers.   Note that each fiber of a locally Moishezon morphism is Moishezon (e.g., \cite[Corollary 16]{Kol22}). 
Thus Theorem \ref{thm-local-moishezon-criterion} and Corollary \ref{cor-local-moishezon-criterion-h20-zero}  provide special cases related to Koll{\'a}r's conjecture on degenerations of Moishezon varieties (\cite[Conjecture 3]{Kol22}) that  the special fiber of a certain nonsmooth family with Moishezon general fiber is Moishezon.
\end{rem}

In the case of a smooth family whose total space is a compact complex manifold in the Fujiki class $\mathscr C$, the
torsion-freeness condition is automatic. This yields the following consequence.

\begin{cor}
\label{thm-fujiki-local-moishezon-one-fiber-condition}
Let $f:Z\to S$ be a smooth family (Definition \ref{defn-smoothfamily}) from  a compact complex manifold $Z$  in the Fujiki class $\mathscr C$, with the fiber over $t\in S$ denoted by $Z_t:=f^{-1}(t)$.  Assume that
there exists a point $s\in S$ such that the restriction map
\[
        H^0(Z,\Omega_Z^2)\to H^0(Z_s,\Omega_{Z_s}^2)
\]
is zero.  Then $f$ is locally
Moishezon. In particular, for each positive integer $m$, the $m$-genus
$P_m(Z_t)$ is independent of $t\in S$.
\end{cor}

\begin{proof}
Since, for each $t\in S$, $Z_t$ is also a compact complex manifold in the Fujiki class $\mathscr C$ (\cite[Lemma 4.6]{Fu78/79}), $Z_t$ satisfies the
$\partial\bar{\partial}$-lemma. Consequently,
$H^i\left(Z_t, \mathbb{C}\right) \to H^i\left(Z_t, \mathcal{O}_{Z_t}\right)$
is surjective for each $i\in \mathbb{N}$ and for each $t \in S$.
Then the torsion-freeness condition in Theorem \ref{thm-local-moishezon-criterion} is valid by Theorem \ref{kollarfree-preli}, or directly by the invariance of Hodge numbers. Consequently, the local Moishezonness follows from Theorem \ref{thm-local-moishezon-criterion}. The invariance of $P_m(Z_t)$ follows from
\cite[Theorem 1.1]{Tk07} or \cite[Main theorem 1.2-\rm{(i)}]{RT22}.
\end{proof}

\subsection{Local Moishezonness of fibrations on hyperfujiki manifolds}\label{ss-lagran}
In this subsection, we give  applications of Theorem \ref{thm-local-moishezon-criterion} to the   study of the local Moishezonness of fibrations from hyperfujiki manifolds.

Clearly, one may easily obtain the following corollary, which  is essentially the bimeromorphic version of Campana’s local projectivity theorem (\cite[Theorem 1.1]{Ca20}).

\begin{cor}
\label{cor-lagrangi-moishe-basenokahler}
 Let $f:Z\to S$ be a fibration from a hyperfujiki manifold $(Z,\sigma_Z)$ (Definition \ref{def-hyperfujiki}), with the fiber over $t\in S$ denoted by $Z_t:=f^{-1}(t)$.
Assume that there exists a
point $s\in S_0$ ($S_0$   is as in Theorem \ref{thm-local-moishezon-criterion}) such that
$\sigma_Z|_{Z_s}=0$. Then  $f$ is locally Moishezon.
\end{cor}

\begin{proof}
Since $Z$ is hyperfujiki, 
$H^0(Z,\Omega_Z^2)=\mathbb C\sigma_Z$ is $1$-dimensional. Therefore the condition
$\sigma_Z|_{Z_s}=0$ implies that the restriction map
\[
        H^0(Z,\Omega_Z^2)\to H^0(Z_s,\Omega_{Z_s}^2)
\]
is zero.
Since $Z$ is hyperfujiki, 
$K_Z\cong \mathcal O_Z$.  
Consequently, it follows from Lemma \ref{lem-torsion-free-hyperfujiki} that $R^2f_*\mathcal O_Z$ is torsion-free.
Thus the application of   Theorem \ref{thm-local-moishezon-criterion} gives that $f$ is locally Moishezon.
\end{proof}

We next show that fibrations  from hyperfujiki manifolds onto lower-dimensional normal compact K\"ahler spaces    must be Lagrangian.  For such fibrations, the restriction-vanishing condition in
Corollary \ref{cor-lagrangi-moishe-basenokahler} is therefore
automatic.

To this end, we first recall the definition of the BBF form
for $\partial\bar{\partial}$-complex symplectic manifolds,
a class that includes (\cite[Definition 1.1]{ACRT18} and
Lemma \ref{hyperfujiki d-closed}) hyperfujiki manifolds.

\begin{defn}[{e.g., \cite[Definition 4.2]{ACRT18}}]\label{def-bbf-form}  Let $(X, \sigma)$ be a $\partial \bar{\partial}$-complex symplectic manifold  of dimension $2 n$.
Normalize $\sigma$ by $\int_X(\sigma\wedge\bar\sigma)^n=1$.
On $H^2(X,\mathbb C)$, define the \emph{Beauville--Bogomolov--Fujiki  form} by
\begin{equation}\label{eq-bbf-definition}
\begin{aligned}
q_X(\alpha):={}&\frac n2\int_X\alpha^2
                       \wedge(\sigma\wedge\bar\sigma)^{n-1}\\
 &+(1-n)
  \left(\int_X\alpha\wedge\sigma^{n-1}\wedge\bar\sigma^n\right)
  \left(\int_X\alpha\wedge\sigma^n\wedge\bar\sigma^{n-1}\right).
\end{aligned}
\end{equation}
In particular, for $a\in H^{1,1}_{dR}(X)$, type considerations give
$q_X(a)=\tfrac n2\int_Xa^2\wedge(\sigma\wedge\bar\sigma)^{n-1}$.  Furthermore, we denote by $B_X(\alpha,\beta):=\tfrac12
(q_X(\alpha+\beta)-q_X(\alpha)-q_X(\beta))$ the  polarization of $q_X$ (so $q_X(u)=B_X(u,u)$).
\end{defn}

We now give a proof of  the Fujiki relation for hyperfujiki manifolds (a more
general Fujiki relation was established in \cite[Theorem 1.2]{KV19}  via a holomorphically symplectic version of Tian--Todorov lemma). The classical hyperk\"ahler case is due to Fujiki (\cite[Theorem 4.7]{Fu87}), whose proof uses the $Sp(1)$-action  and  hyperk\"ahler $3$-space (\cite[Lemmas 4.8–4.10]{Fu87}).
Motivated by the argument of \cite[Proposition 4.4]{CGG22} for establishing the Fujiki relations for  the second Chern class and by \cite[\S 2.1]{OG12} for the application of the  period theory,   here we  use the  period theory (\cite{ACRT18}, \cite{CT18}) for simple  $\partial\bar{\partial}$-complex symplectic manifolds (\cite[Definition 1.1]{ACRT18}), a class that includes hyperfujiki manifolds.

\begin{lemma}\label{lem-fujiki-relation}
If $(X,\sigma)$ is a  hyperfujiki manifold (Definition \ref{def-hyperfujiki}) of dimension $2n$, then we have
\begin{equation}\label{eq-fujiki-relation}
  \int_X\alpha^{2n}=\binom{2n}{n}q_X(\alpha)^n
  \qquad\text{for each }\alpha\in H^2(X,\mathbb C).
\end{equation}
\end{lemma}

\begin{proof}
We divide the proof  into four parts.

\begin{step}\label{step-fujiki-polynomials}
The cup-product polynomial $P$ and the  quadric cut out by $q_X$
\end{step}

Put $V:=H^2(X,\mathbb C)$ and $P(\alpha):=\int_X\alpha^{2n}$. Consider the class $\alpha_0:=[\sigma]+[\bar\sigma]\in V$.
In the binomial expansion of $(\sigma+\bar\sigma)^{2n}$,
only the term $\binom{2n}{n}\sigma^n\wedge\bar\sigma^n$
can be nonzero.
The normalization $\int_X(\sigma\wedge\bar\sigma)^n=1$
then gives $P(\alpha_0)\neq 0.$
Thus $P$ is not identically  zero.  
By multi-linearity of the cup product and linearity of integration, $P$ is a homogeneous polynomial of degree $2n$.

Note that the  hyperfujiki manifold  $X$  satisfies the  $\partial\bar\partial$-lemma.  Then the BBF form $q_X$ is nondegenerate (\cite[Theorem 2.2]{CT18}). 
Furthermore, $X$ admits a big class and thus $h^{1,1}(X)\geq 1$ by the $\partial\bar\partial$-lemma (e.g., \cite[Lemma 5.0.12]{Ba15}). Consequently,   the projective quadric $$Q_\sigma:=\{q_X=0\}\subseteq\mathbb P(V)$$ is irreducible (\cite[Theorem 2.1]{CT18}).  Moreover,  since $q_X$ is homogeneous of degree two, the non-degeneracy gives that 
$q_X$ has no multiple irreducible factors.  Thus its affine cone $$C_\sigma:=\{v\in V:q_X(v)=0\}\subseteq V$$ is a reduced and irreducible hypersurface.

\begin{step}\label{step-fujiki-periods}
Local periods and vanishing of directional derivatives
\end{step}

Let $\pi:\mathcal X\to S$ be a sufficiently small representative
of the smooth universal deformation of $X$, with $X_0=X$
(\cite[Corollary 3.5]{ACRT18}).
After shrinking $S$, each fiber $X_t:=\pi^{-1}(t)$ is a simple
$\partial\bar\partial$-complex symplectic manifold
(\cite[Proposition 3.1]{ACRT18}).
By Ehresmann's theorem, we may choose a local differentiable
trivialization and let $\varphi_t:X\to X_t$ be the induced
diffeomorphisms, with $\varphi_0=\operatorname{id}_X$.
Their pullbacks
$\mu_t:=\varphi_t^*:H^\bullet(X_t,\mathbb C)\to
H^\bullet(X,\mathbb C)$
are graded ring isomorphisms preserving cup products.
For each $t\in S$, choose a holomorphic symplectic form
$\sigma_t$ on $X_t$.
With respect to these isomorphisms, the period map is
$\mathcal P(t)=\mathbb C\,\mu_t([\sigma_t])\in\mathbb P(V)$.

By the  local Torelli Theorem (\cite[Theorem 4.6]{ACRT18}), the period map $\mathcal P$
is a local biholomorphism near $0\in S$.
Its image therefore contains a nonempty open subset
$W\subseteq Q_\sigma$.
Denote the cone of nonzero vectors over $W$  by
\[
\mathcal U
:=
\{v\in V\setminus\{0\}:[v]\in W\}.
\]
Then it follows that each $v\in\mathcal U$
can be written as $v=\lambda\cdot \mu_t([\sigma_t])$
for some $t\in S$ and $\lambda\in\mathbb C^*$.
Since the $(2n+2,0)$-form $\sigma_t^{n+1}$ on $X_t$ vanishes by the type reason, $[\sigma_t]^{n+1}=[\sigma_t^{n+1}]=0$
in $H^{2n+2}(X_t,\mathbb C)$. Since $\mu_t$ preserves cup products, it follows that
$$v^{n+1}=0$$
in $H^{2n+2}(X,\mathbb C)$.

For any fixed $\beta\in V$ and  $v\in\mathcal U$, the binomial formula gives
\[
P(v+\tau\beta)
=
\int_X(v+\tau\beta)^{2n}
=
\sum_{k=0}^{2n}
\binom{2n}{k}\tau^k
\int_Xv^{2n-k}\beta^k,
\]
where $\tau\in\mathbb C$ is an auxiliary parameter.
For each integer $j$ with $0\leq j<n$, 
it follows from $v^{n+1}=0$ that  $v^{2n-j}\beta^j=0$
in $H^{4n}(X,\mathbb C),$ and thus we obtain that the $j$-th  derivative
\begin{equation}\label{j-th-deri-0}
    \frac{d^j}{d\tau^j}P(v+\tau\beta)|_{\tau=0}=j!\binom{2n}{j}\int_Xv^{2n-j}\beta^j=0.
\end{equation}
In particular, taking $j=0$  gives
$$P(v)=0 \text{ for each } v\in\mathcal U.$$

\begin{step}\label{step-fujiki-divisibility}
Divisibility by $q_X^n$
\end{step}

Recall that $C_\sigma$ is irreducible and  $\mathcal U$ is a nonempty open subset of
$C_\sigma$. It then follows that  $P$ vanishes on all of $C_\sigma$.
Since $q_X$ has no multiple irreducible factors, it follows from 
 Hilbert's Nullstellensatz theorem that  $q_X$ divides $P$ in  the polynomial ring
$$A:=\mathbb C[z_1,\cdots, z_N]$$ with $N=\operatorname{dim}_{\mathbb C} V$.      We may write
$$P=q_X^rR$$ in $A$, where $r\geq1$ is maximal such that  $q_X$ does not divide
 $R\in A$.  Then there exists $v\in\mathcal U$
with $R(v)\neq 0$.
As $v\neq 0$ and the polarization $B_X$ of $q_X$ is
nondegenerate, $(dq_X)_v=2B_X(v,\bullet)$ is a nonzero linear
functional on $V$.
We may therefore choose $\beta\in V$ such that
$(dq_X)_v(\beta)=1$.

For the aforementioned $v$ and $\beta$, consider the polynomial $q_X(v+\tau\beta)$ in $\tau$, whose degree is at most $2$. Since $q_X(v)=0$, by direct calculation of the coefficient of this polynomial, we obtain that
\[q_X(v+\tau\beta)= (dq_X)_v(\beta) \tau+ q_X(\beta) \tau^2=\tau\bigl(1+\tau q_X(\beta)\bigr),\]
where the last equality is due to $(dq_X)_v(\beta)=1$.
Thus $P(v+\tau\beta)=\tau^rG(\tau)$, where
$G(\tau):=(1+\tau q_X(\beta))^rR(v+\tau\beta)$
is a polynomial satisfying $G(0)=R(v)\neq0$.
Consequently,
\[
\left.\frac{d^r}{d\tau^r}P(v+\tau\beta)\right|_{\tau=0}
=
r!G(0)
=
r!R(v)
\neq0.
\]
Then \eqref{j-th-deri-0} gives that  $r\geq n$.  Since both $P$ and $q_X^n$
are homogeneous with total degree $2n$, we have $r=n$ and 
$P=cq_X^n$ for some $c\in\mathbb C^*$.

\begin{step}\label{step-fujiki-constant}
Calculation of   the constant
\end{step}

It remains to determine $c$. Calculating the value of $P=cq_X^n$ at $\alpha_0$ and using the   normalization $\int_X(\sigma\wedge\bar\sigma)^n=1$ and the  fact that in calculations only terms proportional to $\sigma^n\wedge\bar\sigma^n$ survive, we conclude that $c=\binom{2n}{n}$. This proves \eqref{eq-fujiki-relation} and completes the proof of Lemma \ref{lem-fujiki-relation}.
\end{proof}

We now prove the following result, whose hyperk\"ahler case is due to
\cite[Theorem 2]{Mat99},
\cite[Theorem 1 and the remark on K\"ahler bases]{Mat01}, and
\cite[Theorem 1]{Mat00}.
Using Lemma \ref{lem-fujiki-relation}, we adapt the argument to the
hyperfujiki setting.  Due to the lack of K\"ahlerness of $X$,  the  new  ingredient in this adaptation is a pointwise positivity
argument for the semipositive pullback of a K\"ahler form on $S$,
which replaces the usage of a K\"ahler class on $X$ in proving that
the general fibers are Lagrangian.

\begin{thm}\label{thm-fujiki-fibration}
Let $f:X\to S$ be a fibration from a hyperfujiki manifold $(X,\sigma)$ (Definition \ref{def-hyperfujiki}) of dimension $2n$ onto a normal compact K\"ahler space $S$.
Assume that $0<\dim S<2n$.
Then  $f$ is a \emph{Lagrangian fibration} in the  sense of Definition \ref{def-Lagran-fibra}:
\begin{enumerate}[\rm{(}i\rm{)}]
\item\label{equ-(1)} $\dim S=n$.
\item\label{lagran-(2)} For each $s\in S$, each irreducible component of
      $(f^{-1}(s))_{\rm{red}}$ is Lagrangian.
\end{enumerate}
\end{thm}

\begin{proof}
We divide the proof into three parts.

\begin{step}
The lower bound of $\dim S$
\end{step}

Put $m:=\dim S$ and choose a K\"ahler form $\omega_S$ on $S$.
Pulling back $\omega_S$
gives a smooth $d$-closed semipositive form $$\eta:=f^*\omega_S$$ on $X$. Note that there is a dense analytic Zariski open subset
$S^\circ\subseteq S_{\rm{reg}}$ over which $f$ is smooth. Note that on $f^{-1}(S^\circ)$, the form $\eta=f^*\omega_S$
has rank $m$ and
\begin{equation}\label{ker-eta-dfx}
    \ker\eta_x=\ker(df_x).
\end{equation}
Indeed, the positivity of $\omega_S$
implies that $\eta_x(u,\bar u)=0$ iff $df_x(u)=0$
for any $u\in T_x^{1,0}X$.
  So we have $$\eta^{m+1}=0$$ on the whole of $X$ by continuity. By the
$\partial\bar\partial$-lemma and the maximal principle on the compact complex manifold $X$, we obtain that 
\[
0\neq a:=[\eta]\in H^{1,1}_{dR}(X,\mathbb R).
\]

Since $m<2n$, we have $a^{2n}=0$. Then Lemma \ref{lem-fujiki-relation}--\eqref{eq-fujiki-relation} gives $q_X(a)=0$. Since $a\neq 0$, we may choose $\beta\in H^2(X,\mathbb C)$ with $B_X(a,\beta)\neq 0$ by the non-degeneracy (\cite[Theorem 2.2]{CT18}) of $B_X$.
The Fujiki relation becomes
\begin{equation}\label{eq-degree-comparison}
 \int_X(\beta+ta)^{2n}
 =\binom{2n}{n}\bigl(q_X(\beta)+2tB_X(a,\beta)\bigr)^n.
\end{equation}
The RHS has degree exactly $n$ in $t$, whereas the
LHS has degree at most $m$  by the vanishing $a^{m+1}=0$.
Thus we have  $$\dim S\geq n.$$

\begin{step}
  The upper bound of $\dim S$ and the Lagrangian property of the general fiber  
\end{step}

Since $q_X(a)=0$, 
the formula  for $q_X$ on $(1,1)$-type classes gives
\begin{equation}\label{vani-inte-eta-sigma}
\int_X\eta^2\wedge(\sigma\wedge\bar\sigma)^{n-1}=0.
\end{equation}

Fix any $x\in X$ and put $$r:=\operatorname{rank}\eta_x.$$
By semipositivity, we can choose local holomorphic coordinates
$z_1,\ldots,z_{2n}$ centered at $x$ such that
\[
 \eta_x=\sqrt{-1}\sum_{j=1}^r(dz_j)_x\wedge(d\bar z_j)_x.
\]
Since $\sigma_x$ is non-degenerate, for each $j=1,\ldots,r$, 
there is a unique
$v_j\in T_x^{1,0}X$ satisfying
$$\sigma_x(v_j,u)=(dz_j)_x(u)$$ for all $u\in T_x^{1,0}X$.
Contracting $(dz_j)_x\wedge\sigma_x^n=0$ with $v_k$ gives
\[
 (dz_j)_x\wedge(dz_k)_x\wedge\sigma_x^{n-1}
   =\frac1n\sigma_x(v_j,v_k)\sigma_x^n.
\]
Expansion therefore yields
\begin{equation}\label{eq-semipositive-squares}
 \eta_x^2\wedge(\sigma_x\wedge\bar\sigma_x)^{n-1}
   =\frac{2}{n^2}\sum_{j<k}|\sigma_x(v_j,v_k)|^2\,
      \sigma_x^n\wedge\bar\sigma_x^n.
\end{equation}
Since $\sigma^n\wedge\bar\sigma^n$ is a positive volume form,
the LHS defines a smooth nonnegative top-degree form
on $X$. Then the vanishing \eqref{vani-inte-eta-sigma} and the skew-symmetry of $\sigma_x$ give that 
\begin{equation}\label{vani-sigma-vj-vk}
    \sigma_x(v_j,v_k)=0
\end{equation}
for any $1\leq j,k\leq r$ at each point $x\in X$.

Thus the linearly independent vectors $v_1,\ldots,v_r$ span
an isotropic subspace of $T_x^{1,0}X$.
Since $(dz_k)_x(v_j)=\sigma_x(v_k,v_j)=0$ for
$1\leq j,k\leq r$, their span lies in
$\bigcap_{k=1}^r\ker(dz_k)_x$, which has complex dimension $2n-r$.
Then $r\leq 2n-r$ and hence we have $$r\leq n.$$
Recall from Step 1 that on $f^{-1}(S^\circ)$, we have $r=m$. So $m\leq n$.
Together with Step 1, this proves the desired  $$\dim S=n$$ in  \eqref{equ-(1)}.

Furthermore,  for any $x\in f^{-1}(S^\circ)$, we have $r=m=n$.  In particular, combining with \eqref{ker-eta-dfx}, we obtain
\[
 \operatorname{Span}_{\mathbb C}\{v_1,\ldots,v_n\}
   \subseteq\bigcap_{j=1}^n\ker(dz_j)_x
   =\ker(df_x).
\]
Both sides have complex dimension $n$, so
$\ker(df_x)=\operatorname{Span}_{\mathbb C}\{v_1,\ldots,v_n\}$.
Since this span is isotropic, $\ker(df_x)$ is Lagrangian.
For each $s\in S^\circ$, we have
$T_x^{1,0}X_s=\ker(df_x)$ at each $x\in X_s$,
so we have
\begin{equation}\label{sigma-Xs-vanish}
    \sigma|_{X_s}=0,
\end{equation}
as desired.

\begin{step}
The Lagrangian property for special fibers    
\end{step}

Since $K_X\cong\mathcal O_X$, Lemma
\ref{lem-torsion-free-hyperfujiki} implies that
$R^2f_*\mathcal O_X$ is torsion-free.
Fix $s\in S$ and a connected Stein neighbourhood $U$ of $s$.
Since $d\sigma=0$, the form $\bar\sigma$ defines a Dolbeault
class on $f^{-1}(U)$.
By \eqref{sigma-Xs-vanish}, this class restricts to zero on each fiber over
$U\cap S^\circ$.  Then
by an argument similar to that in the proof of
Theorem \ref{rela 1-1 class} (from
\eqref{isomorphism by Leray-new} up to the statement that
``$u'$ vanishes everywhere on $S$''),  we obtain
\[
 [\bar\sigma|_{f^{-1}(U)}]=0
 \quad\text{in }H^2(f^{-1}(U),\mathcal O_X).
\]
By the Dolbeault theorem, there exists a smooth $(0,1)$-form
$\zeta$ on $f^{-1}(U)$ such that
\[
 \bar\sigma|_{f^{-1}(U)}=\bar\partial\zeta.
\]

Let $V$ be an irreducible component of $(X_s)_{\rm{red}}$.
Since $V$ is a compact analytic subspace of $X$, it belongs to
Fujiki's class $\mathcal C$
(\cite[Lemma 4.6]{Fu78/79}).
Thus there is a proper modification $\nu:\widetilde V\to V$
with $\widetilde V$ a compact K\"ahler manifold
(\cite[Th\'eor\`eme 3]{Var86}).
Let $j:\widetilde V\to f^{-1}(U)$ be the natural induced map.  Then
$\overline{j^*\sigma}=\bar\partial(j^*\zeta)$.  
Fix any K\"ahler metric on $\widetilde V$, and let
$\bar\partial^*$ denote the formal adjoint of $\bar\partial$. Note that  $\bar\partial^*(j^*\sigma)=0$ for degree reason and hence $\Delta_{\bar\partial}(j^*\sigma)=0$.
Then the K\"ahler identity $\Delta_d=2\Delta_{\partial}=2\Delta_{\bar\partial}$ gives
$\Delta_{\bar\partial}\overline{j^*\sigma}=0$, and hence
$\bar\partial^*\overline{j^*\sigma}=0$. Consequently,
\[
 \bigl\|\overline{j^*\sigma}\bigr\|_{L^2}^2
 =\bigl\langle\bar\partial(j^*\zeta),
              \overline{j^*\sigma}\bigr\rangle_{L^2}
 =\bigl\langle j^*\zeta,
              \bar\partial^*\overline{j^*\sigma}\bigr\rangle_{L^2}
 =0.
\]
That is to say,  $j^*\sigma=0$. Consequently, 
$\sigma$ vanishes on the whole of $V_{\rm{reg}}$ and hence $\dim V\leq n$ by symplectic linear algebra.

Furthermore,  choose a point $x\in V_{\rm{reg}}$ that does not belong
to any other irreducible component of $(X_s)_{\rm{red}}$.
Then   by the fiber dimension formula (e.g., \cite[(1.17)]{PR94}), we have
\[
 \dim V=\dim_x X_s
 \geq \dim_x X-\dim_s S
=n.
\]
Consequently, $\dim V
=n.$
Thus each such $V$ has dimension $n$ and is Lagrangian.   In particular,  $f$ is equidimensional.  This completes the proof of Theorem \ref{thm-fujiki-fibration}.

\end{proof}

Now we are ready to show that the restriction-vanishing condition in Corollary \ref{cor-lagrangi-moishe-basenokahler} is automatic in the following natural setting.

\begin{cor}
\label{cor-lagrangi-moishe}
Let $f:Z\to S$ be a fibration from a hyperfujiki manifold
$(Z,\sigma_Z)$ (Definition \ref{def-hyperfujiki}) onto a
normal compact K\"ahler space $S$  with $0<\dim S<\dim Z$.
Then $f$ is locally Moishezon.
\end{cor}

\begin{proof}
By Theorem \ref{thm-fujiki-fibration}, $f$ is a Lagrangian
fibration.
Choose $s\in S_0$, where $S_0$ is as in
Theorem \ref{thm-local-moishezon-criterion}, and write
$Z_s:=f^{-1}(s)$.
Then $\sigma_Z|_{Z_s}=0$.
Since $H^0(Z,\Omega_Z^2)=\mathbb C\sigma_Z$, the restriction map
\[
 H^0(Z,\Omega_Z^2)\longrightarrow H^0(Z_s,\Omega_{Z_s}^2)
\]
is zero.
Moreover, $K_Z\cong\mathcal O_Z$, so Lemma
\ref{lem-torsion-free-hyperfujiki} implies that
$R^2f_*\mathcal O_Z$ is torsion-free.
Thus Theorem \ref{thm-local-moishezon-criterion} shows that
$f$ is locally Moishezon.
\end{proof}

	\section{Relatively ample line bundles and  local projectivity }\label{s projective-prime}

We now turn from the local Moishezonness question to the local projectivity question, which provides a new proof and a generalization of Campana’s  local projectivity theorem (\cite{Ca20}) for Lagrangian fibrations from  hyperk\"ahler manifolds. The argument parallels, in some sense, the construction in Section \ref{sec-generic-relative-big-line-bundles}, except that the theory of non-K\"ahler loci for verifying bigness is replaced by the singular Demailly-P\u{a}un theorem of Collins–Tosatti for verifying ampleness.

\begin{thm}\label{thm-local-relative-ample-line-bundle}
Let $f:Z\to S$ be a fibration from  a compact K\"ahler manifold $Z$, with the fiber over $t\in S$ denoted by $Z_t:=f^{-1}(t)$. Assume that the following conditions
hold:
\begin{enumerate}[\rm{(}1\rm{)}]
\item\label{condition-zero-one-fiber-ample} there exists a point $s\in S_0$ such
that the restriction map
\[
        H^0(Z,\Omega_Z^2)\to H^0(Z_s,\Omega_{Z_s}^2)
\]
is zero, where $S_0:=D\cap S_{\rm{reg}}$ (as in the proof of Lemma \ref{lem-zero-one-fiber-zero-all-fibers})  with  $S_{\rm{reg}}$  the smooth locus of $S$ and $D$  the maximal  analytic Zariski open subset of  $S$  such that $f$ is flat and submersive over $D$ (note that $S_0$ is analytically Zariski open in $S$);
\item\label{condition-torsion-free-ample} $R^2f_*\mathcal O_Z$ is torsion-free.
\end{enumerate}
Then $f$ is locally projective, i.e.,  for each $t\in S$, there exists an  open neighbourhood
$U_t\subseteq S$ of $t$ and a holomorphic line bundle $L_{U_t}$ on
$f^{-1}(U_t)$
such that $L_{U_t}$ is $f|_{f^{-1}(U_t)}$-ample. 
\end{thm}

As a direct corollary of Theorem \ref{thm-local-relative-ample-line-bundle}, we obtain Campana's local projectivity theorem. 

\begin{cor}[{\cite[Theorem 1.1]{Ca20}}]
Let $f: Z \to S$ be a Lagrangian fibration (Definition \ref{def-Lagran-fibra} ) from a hyperk\"ahler manifold $Z$ onto a   projective variety $S$. Then $f$ is locally projective.
\end{cor}

\begin{proof}
Recall that $\dim H^0(Z,\Omega_Z^2)=1$. By Definitions \ref{def-lagran} and \ref{def-Lagran-fibra}, for each point $s\in S_0$ ($S_0$ is as in Theorem \ref{thm-local-relative-ample-line-bundle}), the restriction map
\[
H^0(Z,\Omega_Z^2)\to H^0(Z_s,\Omega_{Z_s}^2)
\]
is zero. Since $Z$ is K\"ahler and $K_Z$ is trivial, Lemma \ref{lem-torsion-free-hyperfujiki} implies that $R^qf_*\mathcal O_Z$ is torsion-free for each $q\geq 0$. Therefore, $f$ is locally projective by Theorem \ref{thm-local-relative-ample-line-bundle}.
\end{proof}

In the case of a smooth family whose total space is compact K\"ahler, the
torsion-freeness condition is automatic. This yields the following consequence.

\begin{cor}\label{cor-local-relative-ample-line-bundle-smooth-family}
Let $f:Z\to S$ be a smooth family (Definition \ref{defn-smoothfamily}) with $Z$  compact K\"ahler and  each fiber over $t\in S$  denoted by $Z_t:=f^{-1}(t)$.  
Assume that there exists a point
$s\in S$ such that the restriction map
\[
        H^0(Z,\Omega_Z^2)\to H^0(Z_s,\Omega_{Z_s}^2)
\]
is zero.
Then $f$ is locally projective.
\end{cor}

\begin{proof}
    Since $Z$ is K\"ahler, each fiber of $f$ is K\"ahler. It then follows from Theorem \ref{kollarfree-preli} that  $R^i f_* \mathcal{O}_Z$ is locally free  for every $i\in \mathbb{N}$.
     Therefore, $f$ is locally projective by Theorem \ref{thm-local-relative-ample-line-bundle}.
\end{proof}

Note that the restriction-vanishing condition in Theorem \ref{thm-local-relative-ample-line-bundle} is satisfied, for instance, when $H^0(Z_s,\Omega_{Z_s}^2)=0$. In this case, $Z_s$ is projective by Kodaira's projectivity criterion for compact  K\"ahler  manifolds.
Note also that even if each fiber of a smooth family $g:X\to D$ is projective, $g$ need not be locally projective (e.g., \cite[Introduction, Theorem 1-(1.3), Remark (1.5)]{Kol22b}). Thus, even in this special situation, Theorem \ref{thm-local-relative-ample-line-bundle} gives a nontrivial conclusion:

\begin{cor}
Let $f:Z\to S$ be a smooth family (Definition \ref{defn-smoothfamily}) with $Z$ compact K\"ahler and each fiber over $t\in S$ denoted by $Z_t:=f^{-1}(t)$. Assume that there exists a point $s\in S$ such that $H^0(Z_s,\Omega_{Z_s}^2)=0$. Then $f$ is locally projective.
\end{cor}

For examples of Theorem \ref{thm-local-relative-ample-line-bundle}  in which neither $H^0(Z,\Omega_Z^2)$ nor $H^0(Z_s,\Omega_{Z_s}^2)$ vanishes, see Example \ref{ex-product-nonprojective-genus-one-k3} in the appendix.
  Now we prove Theorem \ref{thm-local-relative-ample-line-bundle}.

\begin{proof}[Proof of Theorem \ref{thm-local-relative-ample-line-bundle}]
Note that $S_0$ is a nonempty analytic
Zariski open subset of $S$.
By condition \eqref{condition-zero-one-fiber-ample} and Lemma
\ref{lem-zero-one-fiber-zero-all-fibers}, the restriction map
\[
        H^0(Z,\Omega_Z^2)\longrightarrow H^0(Z_\tau,\Omega_{Z_\tau}^2)
\]
is zero for each $\tau\in S_0$.

\begin{step}\label{step-cons of u and L}
Construction of an integral $2$-class and a holomorphic line bundle
\end{step}

Here we use techniques similar to those in Step 1 of the proof of Theorem \ref{thm-local-moishezon-criterion}. 
 Since $Z$ is compact K\"ahler, it satisfies
the $\partial\bar{\partial}$-lemma. 
 Thus we have the
 decomposition  
\[
        H^2(Z,\mathbb C)
        =
        H^{2,0}_{\rm dR}(Z)\oplus H^{1,1}_{\rm dR}(Z)\oplus H^{0,2}_{\rm dR}(Z)
\]
and the induced continuous projection 
\[
        p_{1,1}:H^2(Z,\mathbb R)\to H^{1,1}_{\rm dR}(Z,\mathbb R),
\]
where the notations are as defined in Section \ref{section preli}.

 Let $\omega_0$ be a K\"ahler form on $Z$. Since the
K\"ahler cone of $Z$ is open and $H^2(Z,\mathbb Q)_{\mathbb R}$ is dense in
$H^2(Z,\mathbb R)$, we may choose a class $a\in H^2(Z,\mathbb Q)$ sufficiently
close to $[\omega_0]$ such that $p_{1,1}(a_{\mathbb R})$ is a K\"ahler class. After multiplying
$a$ by a positive integer, we obtain an integral  class
$$u\in H^2(Z,\mathbb Z)$$
such that $\beta:=p_{1,1}(u_{\mathbb R})$ is a K\"ahler class on $Z$.
Set $\delta:=u_{\mathbb R}-\beta$. Then we have 
 \[
        \delta
        \in
        \bigl(H^{2,0}_{\rm dR}(Z)\oplus H^{0,2}_{\rm dR}(Z)\bigr)\cap H^2(Z,\mathbb R).
\]

Fix a point $t \in S$.  As in the proof of Theorem \ref{thm-local-moishezon-criterion}, choose a connected Stein open neighbourhood $U \subset S$ of $t$, set $Z_U:=f^{-1}(U)$ and $f_U:=f|_{Z_U}: Z_U \to U$. Then there exists a dense analytic Zariski open subset $D_U \subset U$ such that, for each
$\tau\in D_U$, the fiber $Z_\tau$ is smooth and $u|_{Z_\tau}$ is of type $(1,1)$.
By condition
\eqref{condition-torsion-free-ample}, $ R^2(f_U)_*\mathcal O_{Z_U}$
is torsion-free.  Then we  apply Theorem \ref{rela 1-1 class} to $ f_U:Z_U\to U$ to 
obtain a holomorphic line bundle $L_U$ on $Z_U$ such that  $ c_1(L_U)=u|_{Z_U}.$

\begin{step}\label{step-relative-ampleness}
Verification of relative ampleness via the construction of a positively curved metric
\end{step}

It remains to prove that $L_U$ is $f_U$-ample.   
\footnote{In \cite{Ca20}, relative ampleness is verified by using    ``the relative version of Nakai-Moishezon criterion (in the version of Grauert, \cite{G62}, Satz 3, which does not presuppose algebraicity)".}  To get this, we will  use Collins--Tosatti's singular Demailly--P\u{a}un theorem \cite[Theorem 1.1]{CT16} to construct  a positively curved  metric of $L_U$ near any  fiber (possibly after  shrinking $U$).

Let $h_U$ be a smooth Hermitian metric of $L_U$, and let $\alpha$ be its (normalized)
  curvature form such that 
\[
        [\alpha]
        =
        c_1(L_U)_{\mathbb R}
        =
        u_{\mathbb R}|_{Z_U}
        \quad\text{in }H^2(Z_U,\mathbb R).
\]
Choose a K\"ahler form $\omega$ on $Z$ representing the K\"ahler class $\beta$.

Let  $\tau\in U$ be  any fixed point.  We will apply the singular Demailly--P\u{a}un theorem
of Collins--Tosatti \cite[Theorem 1.1]{CT16} to the compact analytic
subvariety $(Z_\tau)_{\rm{red}}\subset Z_U$
of the  K\"ahler manifold $(Z_U,\omega|_{Z_U})$. 

Let $Y\subseteq (Z_\tau)_{\rm{red}}$
be any positive-dimensional irreducible  subvariety, and set
$d:=\dim Y$. Let $\nu:\widetilde Y\to Y$ be a resolution obtained as a finite composition of blow-ups. In particular, $\nu$ is projective.  Denote by
\[
        \mu_U:\widetilde Y\xrightarrow{\nu}Y\hookrightarrow Z_U
\]
the induced holomorphic map, and by
\[
        \mu:\widetilde Y\xrightarrow{\mu_U}Z_U\hookrightarrow Z
\]
the induced holomorphic map to $Z$.

Recall from \cite[II, Proposition 1.3.1]{Vr89}   that projective morphisms are  K\"ahlerian; for a proper K\"ahler morphism, if  the base is  K\"ahler, then the total space over any relatively compact open subset of the base is also K\"ahler. Consequently,   $\widetilde Y$ is a compact K\"ahler manifold, because $Y$ is K\"ahler.

Recall that $ \delta:=u_{\mathbb R}-\beta$.
We first prove that   $\mu^*\delta=0$ in   $H^2(\widetilde Y,\mathbb R)$ (note that $\tau$ is not assumed to be in  $S_0$).

Indeed, since $c_1(L_U)=u|_{Z_U}$, the class
\[
        \mu^*u_{\mathbb R}
        =
        c_1(\mu_U^*L_U)_{\mathbb R}
\]
is of type $(1,1)$. The class $\mu^*\beta$ is also of type $(1,1)$. Thus
$ \mu^*\delta$
is of type $(1,1)$.

On the other hand, by construction $\delta$ is a real linear combination of
the classes $[\operatorname{Re}\eta_j]$ and $[\operatorname{Im}\eta_j]$, where
$\eta_1,\ldots,\eta_m$ is a basis of $H^0(Z,\Omega_Z^2)$. For each $j$, the
pullback $\theta_j:=\mu^*\eta_j$
is a holomorphic $2$-form on $\widetilde Y$. Since $\widetilde Y$ is compact
K\"ahler, the classes
$ [\operatorname{Re}\theta_j],
        [\operatorname{Im}\theta_j]$
belong to
\[
        \bigl(H^{2,0}_{\rm dR}(\widetilde Y)\oplus H^{0,2}_{\rm dR}(\widetilde Y)\bigr)
        \cap H^2(\widetilde Y,\mathbb R).
\]
Since pullback commutes with taking real and imaginary parts, it follows that
\[
        \mu^*\delta
        \in
        \bigl(H^{2,0}_{\rm dR}(\widetilde Y)\oplus H^{0,2}_{\rm dR}(\widetilde Y)\bigr)
        \cap H^2(\widetilde Y,\mathbb R).
\]
Thus $\mu^*\delta$ is simultaneously of type $(1,1)$ and of type
$(2,0)+(0,2)$. By the Hodge decomposition on the compact K\"ahler manifold
$\widetilde Y$, we obtain $ \mu^*\delta=0.$

Consequently,
\[
        [\mu_U^*\alpha]
        =
        \mu^*u_{\mathbb R}
        =
        \mu^*\beta
        =
        [\mu^*\omega]
        \quad\text{in }H^2(\widetilde Y,\mathbb R).
\]
Since all forms involved are closed and $\widetilde Y$ is compact, the following
integrals depend only on their cohomology classes. Thus, for each integer
$k$ with $1\leq k\leq d$, we have
\[
        \int_Y \alpha^k\wedge \omega^{d-k}
        =
        \int_{\widetilde Y}
        \mu_U^*\alpha^k\wedge \mu^*\omega^{d-k}
        =
        \int_{\widetilde Y}
        \mu^*\omega^d
        >0.
\]

In conclusion, for each positive-dimensional irreducible compact analytic subvariety
$Y\subseteq (Z_\tau)_{\rm{red}}$ and each integer $k$ with $1\leq k\leq \dim Y$, one has
\[
        \int_Y \alpha^k\wedge \omega^{\dim Y-k}>0.
\]
Therefore the application of \cite[Theorem 1.1]{CT16} gives that there exist an open neighbourhood
$W_\tau\subseteq Z_U$ of $(Z_\tau)_{\rm{red}}$ (and thus of $Z_\tau$, since both of them  have the same underlying topological space)  and a smooth function
$\varphi_\tau$ on $W_\tau$ such that
\[
        \alpha+\frac{\sqrt{-1}}{\pi}\partial\bar\partial\varphi_\tau
\]
is a K\"ahler form on $W_\tau$.

We  define a new Hermitian metric $h'_U$ on $L_U|_{W_\tau}$ by
\[
        |\bullet|_{h'_U}^2
        :=        |\bullet|_{h_U}^2 e^{-2\varphi_\tau}.
\]
Now the curvature  is $$c_1(L_U,h'_U)=\alpha+\frac{\sqrt{-1}}{\pi}\partial\bar\partial\varphi_\tau,$$ which 
 is a K\"ahler form on $W_\tau$. That is to say, 
$L_U|_{W_\tau}$ has  strictly positive curvature. 
  In particular,  $L_U|_{(Z_{\tau})_{\rm{red}}}$ is
positive.  Consequently, it follows from Grauert's embedding theorem (e.g.,
\cite[p.343, proof of Satz 2]{G62} or \cite[Proposition 2.4]{CMM17}) for reduced
compact complex spaces that $L_U|_{(Z_{\tau})_{\rm{red}}}$  is ample.   Recall from Lemma \ref{lemma-analytic-ample-reduction} that the ampleness of a holomorphic line bundle on a complex space can be checked on its reduction.  Then $L_U|_{Z_{\tau}}$ is ample.  Since $\tau\in U$ is arbitrary,  we obtain that $L_U$
is $f_U$-ample (Definition \ref{def-proj morp}).
This completes the proof.
\end{proof}

\appendix

\section{Auxiliary results and examples}

The following result should be well known, although we were unable to find a suitable reference.   We therefore include a proof  for the reader’s convenience.

\begin{lemma}[= Lemma \ref{lem-big-chern-big-line}]
\label{lem-big-chern-big-line-appen}
Let $M$ be a compact complex manifold in the Fujiki class $\mathscr C$ and let $L$ be a holomorphic line bundle
on $M$.  If $(c_1(L))_{\mathbb R}$ 
 is a big class, then $L$ is a big line bundle.
\end{lemma}

\begin{proof}
Since $(c_1(L))_{\mathbb R}$ is big, it admits a K\"ahler current $T$, i.e., there exist a real $d$-closed $(1,1)$-current $T$
representing $(c_1(L))_{\mathbb R}$, a Hermitian form $\omega$ on $M$, and a
constant $\varepsilon>0$ such that
$T\geq \varepsilon\omega$.   By Theorem \ref{ji-shiff-big}, for proving this lemma, it suffices to show that $L$ admits a singular Hermitian metric $h$ (in the sense of Demailly) such that
the Chern curvature current $c_1(L,h)$ is equal to $T$.

Choose a smooth Hermitian metric $h_0$  on $L$ and set $\theta_0=c_1(L,h_0).$  Since $M$ is a compact complex manifold in the Fujiki class $\mathscr C$, it satisfies $\partial\bar\partial$-lemma.
Then
there exists (e.g., \cite[Lemma 5.0.12]{Ba15}) a real distribution $\varphi$ on $M$ such that
\[
T=\theta_0+\frac{\sqrt{-1}}{\pi}\partial\bar\partial\varphi .
\]
Moreover, the inequality $T\geq \varepsilon\omega$ gives
\[
\frac{\sqrt{-1}}{\pi}\partial\bar\partial\varphi
\geq
\varepsilon\omega-\theta_0 .
\]
Since $\varepsilon\omega-\theta_0$ is smooth, it follows that, on each
coordinate ball, there is a constant $C>0$ such that
\[
\frac{\sqrt{-1}}{\pi}\partial\bar\partial\varphi
\geq
-C\frac{\sqrt{-1}}{\pi}\sum_j dz_j\wedge d\bar z_j .
\]
Thus $\varphi+C|z|^2$ is locally plurisubharmonic (i.e., $\varphi$ is quasi-plurisubharmonic). In particular $\varphi\in L^1_{\rm loc}(M)$.  Consequently, $h=h_0e^{-2 \varphi}$ is
a singular Hermitian metric  (in the sense of Demailly)  of $L$.
 Its curvature current is
\[
c_1(L,h)
=
c_1(L,h_0)+\frac{\sqrt{-1}}{\pi}\partial\bar\partial\varphi
=
\theta_0+\frac{\sqrt{-1}}{\pi}\partial\bar\partial\varphi
=
T.
\]
This completes the proof.
\end{proof}

The next lemma shows that the ampleness of a holomorphic line bundle on a complex analytic space can be checked on its reduction. Its proof is identical to that in the scheme setting (e.g., \cite[Proposition 1.2.16]{Lz04a}), using the analytic cohomological criterion for ampleness (e.g., \cite[Chapter IV, Theorem 4.1]{BS76}) and the analytic Serre vanishing theorem (e.g., \cite[Chapter IV, Theorem 2.1]{BS76}).

\begin{lemma}[{Analytic version of \cite[Proposition 1.2.16]{Lz04a}}]\label{lemma-analytic-ample-reduction}
Let $X$ be a compact complex analytic space, let $\mathcal{L}$ be an invertible
sheaf, and let
$\iota:X_{\rm red}\hookrightarrow X$ be the reduction map. Put
$\mathcal{L}_{\rm red}=\iota^{*}\mathcal{L}$. Then $\mathcal{L}$ is
ample on $X$ if and only if $\mathcal{L}_{\rm red}$ is ample on
$X_{\rm red}$. \footnote{Note that an ample line bundle on $X_{\rm red}$ may not be lifted to $X$ and the projectivity of $X_{\rm red}$ may not imply the projectivity of $X$.}
\end{lemma}

We now 
illustrate the local Moishezonness criterion of
Theorem \ref{thm-1.1} (= Theorem  \ref{thm-local-moishezon-criterion}) and give an example of a Lagrangian fibration $f:\widetilde X^+\to B$ satisfying the
conditions of Theorem \ref{thm-1.1},  where $\widetilde X^+$   is  a $4$-dimensional non-K\"ahler  hyperfujiki manifold and  $B\cong\mathbb P^2$. This example is based on 
Namikawa's  construction (\cite[Example (1.7)(ii), pp. 608--609]{Nam01}).

\begin{ex}\label{prop-kummer-flop}
First, we recall   that 
Namikawa \cite[Example (1.7)(ii), pp. 608--609]{Nam01}
constructed a smooth  non-K\"ahler holomorphic symplectic
$4$-fold $\widetilde X^+$ by a Mukai flop  of
$\widetilde X:=\operatorname{Hilb}^2(S)$ which is hyperk\"ahler (\cite[\S 6, Theorem 3 and the preceding paragraph, p. 768]{Bea83}), as in the following  diagram:
\[
\begin{tikzcd}[row sep=1.2em, column sep=2.5em]
& W
  \arrow[dl, "b"']
  \arrow[dr, "b^{+}"]
& \\
\widetilde X
  \arrow[dr, "\pi"']
& &
\widetilde X^{+}
  \arrow[dl, "\pi^{+}"]
\\
& X &
\end{tikzcd}
\]
Here, $S$ is a projective K3 surface equipped with an elliptic fibration $g:S\to\mathbb P^1$ having two fibers of type $I_3$.   
Denote one of the two fibers of type $I_3$  by $g^{-1}(u)=C=C_1+C_2+C_3$, and let
$\mu:\widetilde X\to\operatorname{Sym}^2(S)$ be the
natural bimeromorphic projective morphism
(see also \cite[\S 6]{Bea83}).
For $1\leq i\leq 3$, let $F_i$ be the strict transform of
$\operatorname{Sym}^2(C_i)$ under $\mu$.
The small contraction $\pi$ sends these disjoint planes to
points in a compact complex space $X$, and $\pi^+$ is its flop.
The map $b$ blows up their union, while $b^+$ contracts each
exceptional divisor $b^{-1}(F_i)$ along its other
$\mathbb P^1$-ruling to obtain the manifold $\widetilde X^+$ (e.g., \cite[Definition 3.1 and the preceding paragraphs]{Huy97}) which is non-K\"ahler (\cite[Paragraph -1, p. 608]{Nam01}).

Since $S$ is projective, $\widetilde X=\operatorname{Hilb}^2(S)$ is projective and hence $\widetilde X^+$ is Moishezon.
 Furthermore, $\widetilde X^+$ is simply connected (e.g., \cite[Corollary 2.4]{BP21}).
Hartogs' theorem gives
$H^0(\widetilde X^+,\Omega_{\widetilde X^+}^2)
=\mathbb C\sigma_{\widetilde X^+}$,
where $\sigma_{\widetilde X^+}$ is the symplectic form induced by  the flop from the  hyperk\"ahler manifold  $\widetilde X$. Thus $\widetilde X^+$ is a non-K\"ahler  hyperfujiki manifold
(Definition \ref{def-hyperfujiki}).

Second, 
   we construct a holomorphic map $f$  from $\widetilde X^+$.
Let
\[
g^{(2)}:\operatorname{Sym}^2(S)
\longrightarrow\operatorname{Sym}^2(\mathbb P^1),
\qquad
[x]+[y]\longmapsto[g(x)]+[g(y)],
\]
where $[x]$ denotes the effective zero-cycle supported at $x$
with multiplicity one, and repetitions are allowed.
Set
\begin{equation}\label{iso-p2}
h:=g^{(2)}\circ\mu:\widetilde X\longrightarrow
B:=\operatorname{Sym}^2(\mathbb P^1)\cong\mathbb P^2.
\end{equation}

Note that $h\circ b$ is constant on each exceptional divisor
$b^{-1}(F_i)$, and hence on each fiber of $b^+$.
The properness of $b^+$ and the isomorphism
\[
(b^+)_*\mathcal O_W\cong\mathcal O_{\widetilde X^+}
\]
therefore give a unique holomorphic map
\[
f:\widetilde X^+\longrightarrow B
\]
such that $f\circ b^+=h\circ b$ as in the following diagram.
\[
\begin{tikzcd}[row sep=1.2em, column sep=2.5em]
& W
  \arrow[dl, "b"']
  \arrow[dr, "b^{+}"]
& \\
\widetilde X
  \arrow[dr, "h"']
& &
\widetilde X^{+}
  \arrow[dl, "f"]
\\
& B &
\end{tikzcd}
\]

 Finally, we prove that $f$ is a  fibration
 satisfying  the hypotheses of Theorem \ref{thm-1.1}.
Clearly, $f$ is proper and surjective. 
Let $U\subset\mathbb P^1$ be the nonempty open subset over
which $g$ is smooth, and set
\[
B^\circ
:=
\{[a]+[b]\in B\mid a,b\in U,\ a\neq b\}.
\]
Choosing disjoint small discs around $a$ and $b$ identifies
$h$ locally over $B^\circ$ with the product of the corresponding
smooth elliptic families.
Since the flop is supported over $2[u]\notin B^\circ$, it does
not affect these fibers.
Consequently, $f$ is smooth over $B^\circ$, with
\[
\widetilde X_s^+:=f^{-1}(s)\cong S_a\times S_b,
\qquad
s=[a]+[b]\in B^\circ,
\]
where $S_\bullet:=g^{-1}(\bullet)$.
Then the Stein factorization   shows that all fibers
of $f$ are connected and hence  $f$ is a fibration. Consequently, $f$ is a
Lagrangian  fibration by 
Theorem \ref{thm-fujiki-fibration}. In particular,
the restriction map
\[
H^0(\widetilde X^+,\Omega_{\widetilde X^+}^2)
\longrightarrow
H^0(\widetilde X_s^+,\Omega_{\widetilde X_s^+}^2)
\]
is zero for each $s\in B^\circ$.
Moreover, $K_{\widetilde X^+}\cong\mathcal O_{\widetilde X^+}$,
so Lemma \ref{lem-torsion-free-hyperfujiki} gives that
$R^2f_*\mathcal O_{\widetilde X^+}$ is torsion-free.
Thus $f$ satisfies the hypotheses of Theorem \ref{thm-1.1}.

\end{ex}

We now  give an example of  Theorem  \ref{thm-intro-relative-ample} (= Theorem \ref{thm-local-relative-ample-line-bundle})  in which neither $H^0(Z,\Omega_Z^2)$ nor $H^0(Z_s,\Omega_{Z_s}^2)$ vanishes.  
The construction is based on standard facts on complex  K$3$ surfaces. 


\begin{ex}\label{ex-product-nonprojective-genus-one-k3}

First, we construct two elliptic fibrations  $\pi_1$ and $\pi_2$.
Let $X_1$ and $X_2$ be complex K3 surfaces such that  $\operatorname{Pic}(X_i)=\mathbb Z[F_i]$ and  $  F_i^2=0$, 
where $F_i$ is a non-zero  primitive generator (\cite[Chapter 17, Remark 1.1]{Huy16}). 
Since $\operatorname{Pic}(X_i)$ contains no class of positive
self-intersection, $X_i$ is a non-projective K\"ahler manifold (\cite[Chapter 17, Remark 1.2]{Huy16}).

Note that $K_{X_i}\cong\mathcal O_{X_i}$. Then it follows from the Hirzebruch--Riemann--Roch formula and the Serre duality that 
$$h^0(X_i, F_i)+h^0\left(X_i, -F_i\right) \geq 2.$$  
Recall that $F_i$ is trivial if and only if both $F_i$ and its dual
$-F_i$ admit non-trivial global sections. Hence after replacing the generator $F_i$ by $-F_i$ if necessary, we may choose
$F_i$ to be effective such that $h^0(X_i,-F_i)=0$ and   $ h^0(X_i,F_i)\geq 2.$

 A standard argument shows that each $F_i$ is base point free and that the Kodaira map associated with $|F_i|$ is a fibration $\pi_i: X_i\to\mathbb P^1$ whose general fiber is a smooth curve of genus one.

Second, we construct a fibration satisfying  the hypotheses of Theorem \ref{thm-local-relative-ample-line-bundle}.  
Define
\[
        Z:=X_1\times X_2,
        S:=\mathbb P^1\times\mathbb P^1   \text{ and } f:=\pi_1\times\pi_2:Z\to S.
\]
Then $Z$ is a compact  non-projective K\"ahler manifold and $f$ is a fibration. Since
$H^0(X_i,\Omega_{X_i}^1)=0$,  the K\"unneth formula gives
$$H^0(Z,\Omega_Z^2)\cong \operatorname{pr}_1^*H^0(X_1,\Omega_{X_1}^2)
\oplus
\operatorname{pr}_2^*H^0(X_2,\Omega_{X_2}^2).$$
Consequently, $H^0(Z,\Omega_Z^2)$ is $2$-dimensional.  
Note that
\[
        K_Z
        \cong
        \operatorname{pr}_1^*K_{X_1}\otimes
        \operatorname{pr}_2^*K_{X_2}
        \cong
        \mathcal O_Z.
\]
By Lemma \ref{lem-torsion-free-hyperfujiki}, the torsion-freeness
hypothesis in Theorem \ref{thm-local-relative-ample-line-bundle} is
satisfied in the present situation.

Let $S_0\subset S$ be the open subset consisting of points
$s=(s_1,s_2)$ such that both fibers $(X_1)_{s_1}$ and $(X_2)_{s_2}$ are
smooth. For $s\in S_0$, write $E_i:=(X_i)_{s_i}.$
Then  $ Z_s=E_1\times E_2$.
Since  $E_i$ are curves, $\Omega_{E_i}^2=0$. Applying the K\"unneth
formula again gives
\[
        H^0(Z_s,\Omega_{Z_s}^2)
        \cong
        H^0(E_1,\Omega_{E_1}^1)
        \otimes
        H^0(E_2,\Omega_{E_2}^1).
\]
Since  $E_i$ are of genus-one, $h^0(E_i,\Omega_{E_i}^1)=1$.
Therefore  $h^{2,0}(Z_s)=1.$

 Let $\sigma_i$ be a
generator of $H^0(X_i,\Omega_{X_i}^2)$. Then  $  \operatorname{pr}_1^*\sigma_1,
        \operatorname{pr}_2^*\sigma_2$
form a basis of $H^0(Z,\Omega_Z^2)$. Since $E_i$ is a curve, the
pullback of any holomorphic $2$-form to $E_i$ is zero. Thus  $\operatorname{pr}_1^*\sigma_1|_{Z_s}=0$ and  $\operatorname{pr}_2^*\sigma_2|_{Z_s}=0.$
Consequently, the restriction map
\[
        H^0(Z,\Omega_Z^2)\to H^0(Z_s,\Omega_{Z_s}^2)
\]
is zero for each $s\in S_0$.

In conclusion, $f$ satisfies the hypotheses of Theorem
\ref{thm-local-relative-ample-line-bundle}.   Furthermore, neither $H^0(Z,\Omega_Z^2)$ nor $H^0(Z_s,\Omega_{Z_s}^2)$ vanishes.  

\end{ex}

	\section*{Acknowledgement}
	
The author would like to express his gratitude to  Professors Sheng Rao and  I-Hsun Tsai for their many valuable discussions on topics related to this paper over the years.  He also extends his thanks to Professor  Fr\'{e}d\'{e}ric Campana for answering  a question on \cite{Ca20}.
He is also deeply grateful to the referee for a careful reading
and insightful comments and suggestions, which have greatly
improved both the mathematical content and the presentation of
this paper and led in particular to 
Subsection \ref{ss-lagran} and the examples in the appendix.

\end{document}